\documentclass{amsart}
 
\usepackage[utf8]{inputenc}
\usepackage[T1]{fontenc}
\usepackage{float}
\usepackage{mathtools}
\usepackage{color}
\usepackage[english]{babel}
\usepackage{amsmath} 
\usepackage{amsfonts}
\usepackage{thmtools}
\usepackage{amssymb}
\usepackage{mathrsfs}
\usepackage{amsthm}
\usepackage{mathrsfs}
\usepackage{setspace}
\usepackage{graphicx}
\usepackage{bm}
\usepackage[all]{xy}
\usepackage{lscape}
\usepackage{latexsym}
\usepackage{epigraph}
\usepackage[Sonny]{fncychap}
\usepackage{tikz} 
\usepackage{tikz-cd}
\usetikzlibrary{matrix}
\usepackage{epsfig}
\usepackage{fancyhdr}
\usepackage{hyperref}
\usepackage{paralist}
\usepackage[mathcal]{eucal}
\usepackage{xcolor}
\usepackage{caption}
\usepackage{subcaption}
\usepackage{float}
\numberwithin{equation}{section}

\newtheorem{theorem}{Theorem}[section]
\newtheorem{proposition}[theorem]{Proposition}
\newtheorem{lemma}[theorem]{Lemma}
\newtheorem{corollary}[theorem]{Corollary}

\newtheorem*{theorem*}{Theorem}
\theoremstyle{definition}
\newtheorem{definition}[theorem]{Definition}
\newtheorem{example}[theorem]{Example}
 
\theoremstyle{remark}
\newtheorem{remark}[theorem]{Remark}

\theoremstyle{definition}
\newtheorem*{definition*}{Definition}

\newtheorem*{corollary*}{Corollary}

\definecolor{darkred}{rgb}{0,0,0} 
\definecolor{darkgreen}{rgb}{0,0,0}
\definecolor{darkblue}{rgb}{0,0,0}

\begin{document}

\title{Meridional rank of whitehead doubles} 
 
\author{Ederson R. F. Dutra} \thanks{Research supported by FAPESP, São Paulo Research Foundation, grants  2018/08187-6 and 2021/12276-7.}

\address{Universidade Federal de São Carlos,  São Carlos, Brazil}
 
\email{edersondutra@dm.ufscar.br}

\maketitle

 
\begin{abstract} 
We prove that the meridional rank and the bridge number of the Whitehead double of a prime algebraically tame knot  coincide.  Algebraically tame knots are a broad generalization of torus knots  and iterated cable knots.   
\end{abstract}


\section{Introduction}

Let $\mathfrak k\subseteq \mathrm S^3$ be a knot. The \emph{bridge number of $\mathfrak k$}, denoted by $b(\mathfrak k)$, is defined as the minimal number of bridges in a bridge presentation of $\mathfrak k$.  A    \emph{meridian} of $\mathfrak k$ is an element of $\pi_1(\mathrm S^3\setminus \mathfrak k)$ which is represented by a curve freely homotopic to the boundary path of a disk in $\mathrm S^3$ that intersects $\mathfrak k$ in exactly one point.  The \emph{meridional rank  of  $\mathfrak k$}, denoted by  $w(\mathfrak k)$,  is defined as the minimal number of meridians needed to generate $\pi_1(\mathrm S^3\setminus\mathfrak k)$. It is known that any bridge presentation of $\mathfrak k$ containing $b$ bridges  yields a generating set of $\pi_1(\mathrm S^3\setminus\mathfrak k)$ consisting of $b$ meridians.   Therefore we always have   the inequality   $w(\mathfrak k) \le  b(\mathfrak k)$.

It was asked by S. Cappell and J. Shaneson~\cite[Problem~1.11]{Kirby}  as well as by K. Murasugi  whether  the equality $b(\mathfrak k)=w(\mathfrak k)$  always holds. To this day no counterexamples are known but the equality has been verified  for many classes of knots (using a variety of techniques): knots  whose group is generated by two meridians~\cite{Boi6},   (generalized) Montesinos links~\cite{Boi3, LM},  torus knots~\cite{RostZ}, iterated cable knots~\cite{Co, CH}, knots whose exterior is a graph manifold~\cite{BDJW}, links of meridional rank $3$ whose double branched covers are graph manifolds~\cite{BJW}, and recently for twisted and  arborescent links~\cite{Baader, Misev}.

An interesting result proved recently by R. Blair, A. Kjuchukova, R. Velazquez, and P. Villanueva~\cite{Blair} states that the bridge number of any knot equals its Wirtinger number, an invariant closely related to the meridional rank. This result therefore establishes a weak version of Cappell and Shaneson’s question, and also points to new approaches  to this problem.

In this article we show that the conjuncture holds for Whitehead doubles of prime algebraically tame knots and we also show that the class of   prime algebraically tame knots is closed under braid satellites. 
\begin{theorem}{\label{thm01}}
Let $\mathfrak k$ be a Whitehead double of a prime algebraically tame knot. Then  it holds $w(\mathfrak k)=b(\mathfrak k)$. 
\end{theorem}
 
\begin{theorem}{\label{thm02}}
The class of prime  algebraically tame knots is closed under taking satellites with  braid pattern.
\end{theorem}
  
\subsection*{Acknowledgments}
I would like to thank  the faculty of mathematics of the University of Regensburg  for the hospitality during my one year visit as a pos-doctoral guest in  Regensburg, in special to prof. Stefan Friedl and Filip Misev  for helpful conversations during the preparation of this article and for  
Claudius Zibrowius for a nice office atmosphere.


\section{Basic definitions}{\label{section:basicdef}}

In this section we  fix the notation needed to prove    Theorem~\ref{thm01} and Theorem~\ref{thm02}. We closely follow the notation from~\cite{BZ}.  Let $\mathfrak k \subseteq \mathrm S^3$  be a knot.  The knot exterior  of $\mathfrak k$ is defined as  $$E(\mathfrak k):= \mathrm S^3 \setminus  Int \  V (\mathfrak k)$$  where $V(\mathfrak k)$ is a regular neighborhood of $\mathfrak k$. Let $x_0\in \partial E(\mathfrak k)=\partial V(\mathfrak k)$. The group of $\mathfrak k$ is defined as $G(\mathfrak k):=\pi_1(E(\mathfrak k), x_0)$ and the peripheral  subgroup of $\mathfrak k$ is defined as 
 $P(\mathfrak k)=\pi_1(\partial E(\mathfrak k), x_0).$

A meridian of $\mathfrak k$  was  previously     defined   as an element of $G(\mathfrak k)$ represented by a curve freely homotopic to the boundary path of a disk which intersects $\mathfrak k$ in exactly one point. For our purposes, however,   we need a fixed meridian.

Let  $\mu\subset \partial V(\mathfrak k)$ be the boundary of  a meridional disk of  $ V(\mathfrak k)$  such  that $x_0\in \mu$.  Define $m:=[\mu] \in P(\mathfrak k)$. Thus $m$  is a meridian of $\mathfrak k$  as defined previously  and  any other meridian of $\mathfrak k$ is conjugate to $m$.

Let  also $\lambda \subseteq \partial E(\mathfrak k)$ be a longitudinal curve of $\mathfrak k$, i.e.~a simple closed curve that bounds an orientable surface in $E(\mathfrak k)$,  such that  $x_0 =\lambda\cap \mu$. Define $l:=[\lambda]\in P(\mathfrak k)$. 

Observe that $\{m, l\}$ is  a basis for $P(\mathfrak k)$.


\subsection{Satellite knots}{\label{subsect:satellite}} We recall the construction of satellite knots. Let $\mathfrak k_1\subseteq \mathrm S^3$ be a non-trivial knot.  Let $V\subseteq \mathrm S^3$ be a standardly embedded  unknotted  solid torus  and let  $\mathfrak k_0$ a knot contained in  the interior of $V$ such that $\mathfrak k_0$ intersects each meridional disk of $V$ in  at least one point. For any  homeomorphism  $h$ from $V$ onto a regular neighborhood $ V(\mathfrak k_1)$  of $\mathfrak k_1$  we call   the knot  ${\mathfrak k}_h:=h(\mathfrak k_0)$ a \emph{satellite  knot  with companion $\mathfrak k_1$ and  pattern $(V, \mathfrak k_0)$}, see Figure~\ref{figwh}.
\begin{figure}[h!] 
\begin{center}
\includegraphics[scale=1]{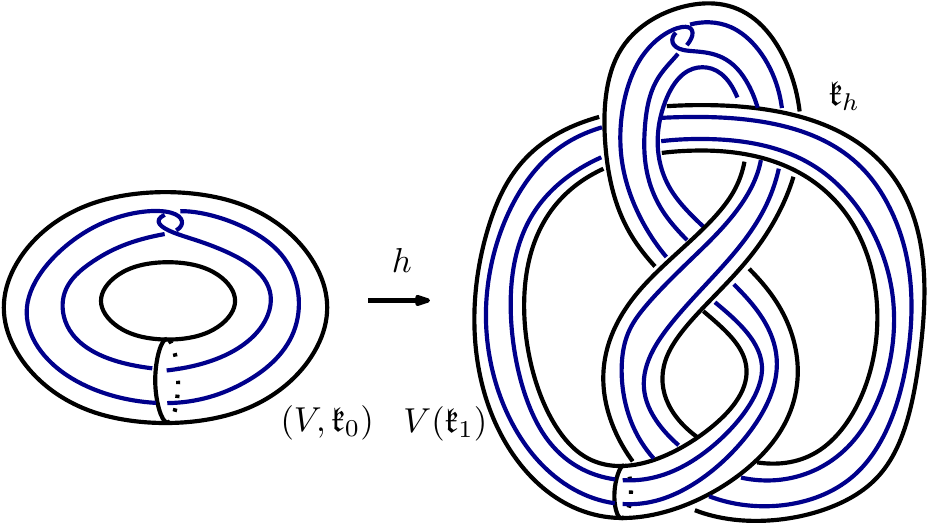}
\caption{Whitehead double of the figure-eight knot.}{\label{figwh}}
\end{center}
\end{figure}
We will be specially concerned with the following classes of satellite knots:
\begin{enumerate}
\item[•] If $(V, \mathfrak k_0)$ is the pattern described in Figure~\ref{figwh},   then  $\mathfrak k_h$  is called  a \emph{Whitehead double} of $\mathfrak k_1$. 

\item[•] If  $\mathfrak k_0$ lies in  a torus $T\subseteq Int(V)$ which is  parallel to $\partial V$,  then $\mathfrak k_h$ is called a \emph{cable knot on $\mathfrak k_1$}. 	

\item[•] If $\mathfrak k_0$ is a  closed braid  standardly   embedded in   $V$  as shown in Figure~\ref{fig:braid}, i.e. any meridional disk $p\times D^2\subseteq  V$ intersects $\mathfrak k_0$ in exactly $n$ points, then $\mathfrak k_h$ is called a  \emph{satellite  knot  with braid pattern}.  
\end{enumerate} 
\begin{figure}[h] 
\begin{center}
\includegraphics[scale=1]{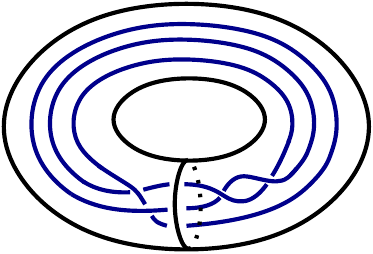}
\end{center}
\caption{$\mathfrak k_0$  is the closed braid $\hat{\beta}$ where  $\beta=\sigma_2{\sigma_1}^{-1}\sigma_2^2$.}\label{fig:braid}
\end{figure}

The exterior of a satellite knot $\mathfrak k=h(\mathfrak k_0)$ clearly decomposes as
$$E(\mathfrak k) =   h(V\setminus Int \ V(\mathfrak{k}_0) ) \cup E(\mathfrak k_1)  \ \text{ and } \  h(V\setminus Int \ V(\mathfrak{k}_0)) \cap E(\mathfrak k_1)=h(\partial V) = \partial E(\mathfrak k_1)$$  
where $V(\mathfrak{k}_0)\subseteq Int(V)$ is a regular neighborhood of $\mathfrak k_0$ in $V$.  The Theorem of Seifert and van-Kampen  implies that  
$$G(\mathfrak k)= \pi_1(V \setminus Int \ V(\mathfrak k_0)) \ast_{(\alpha, C, \omega)} G(\mathfrak k_1)$$
where
\begin{enumerate}
\item[•] $C :=\langle m_e, l_e  \rangle\cong \mathbb Z\oplus \mathbb Z$.

\item[•] $\alpha:C\rightarrow \pi_1(V\setminus Int \ V(\mathfrak k_0))$ is given by 
$$\alpha(m_e)=m_V \ \text{ and } \ \omega(l_e)=l_V$$ where $m_V$  is a meridian of $V$ and $l_V$ is a  meridian of the solid torus $\mathrm S^3\setminus Int \ V$ (and therefore a longitude of $V$). 

\item[•] $\omega:C \rightarrow G(\mathfrak k_1)$ is defined by 
$$\omega(m_e)=h_{\ast}(m_V) \ \text{ and } \  \omega(l_e)=h_{\ast}(l_V)$$ 
where $h_{\ast}$ is the induced isomorphism $\pi_1(\partial V)\to  P(\mathfrak k_1) $. 

\end{enumerate}  
Observe that we can always take $h_{\ast}(m_V)$ as the fixed meridian of $\mathfrak k_1$  and therefore denote $h_{\ast}(m_V)$  simply by $m_1$.

The main tool that we are going to use in this paper  is the theory of  folding in graph of groups developed in \cite{RW}  which is a slightly variation  of the theory developed in   \cite{KMW}. For this reason,    we will always   consider the amalgamated free product  $$G(\mathfrak k)= \pi_1(V \setminus Int \ V(\mathfrak k_0)) \ast_{(\alpha, C, \omega)} G(\mathfrak k_1)$$ as a graph of groups $\mathbb A$ having  a pair of vertices $v_0$ and $v_1$ such that $$A_{v_0}=\pi_1(V\setminus Int \ V(\mathfrak k_0)) \  \ \text{ and } \  \ A_{v_1}=G(\mathfrak k_1)$$  and a single edge pair $\{e, e^{-1}\}$ with $\alpha(e)=v_0$ and $\omega(e)=v_1$   such that $$A_{e}=A_{e^{-1}}:=C=\langle m_e, l_e\rangle.$$ 
The boundary monorphisms $\alpha_e:A_e\to A_{v_0}$ and $\omega_e:A_e\to A_{v_1}$ coincide with  the monorphisms defined before, that is,    $\alpha_e:=\alpha$ and $\omega_e:=\omega$, see Figure~\ref{fig:graphofgroupssatellite}.  
\begin{figure}[h!] 
\begin{center}
\includegraphics[scale=1]{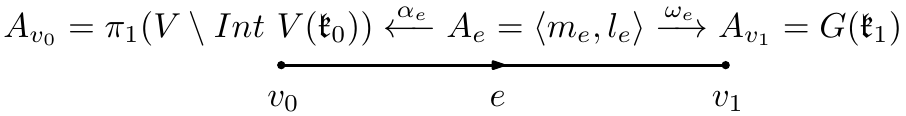}
\caption{The graph of groups $\mathbb A$.}{\label{fig:graphofgroupssatellite}}
\end{center}
\end{figure}

 
\subsection{The group of the braid space and cable space}{\label{sub:sattelite}} Assume that  $\mathfrak k_0 \subset V$ is a closed braid, i.e.~$\mathfrak k_0=\hat{\beta}$  where $\beta$ is an $n$-braid whose  associate permutation $\tau\in S_n$ is  cycle of length $n$. The \textit{braid space of $\beta$} is defined as 
$$BS(\beta):=V\setminus Int \ V(\hat{\beta})$$ 
where $V(\hat{\beta})$ is a regular neighborhood of $\mathfrak k_0=\hat{\beta}$. We can construct $BS(\beta)$ as a mapping torus in the following way. Let $D^2=\{z\in \mathbb C \ | \ \Vert z\Vert \leq 1\}$ and  $$\Sigma=D^2\setminus Int( \delta_1\cup  \ldots \cup  \delta_n)$$  where $\delta_1$ is a small disk in the interior of $D^2$,   $\delta_i=\rho_0^{i-1}(\delta_1)$ for $2\le i\le n$  and where  $\rho_0:D^2\to D^2$ is a rotation  by an angle of $2\pi / n$ about the origin $x_0:=0$.  It is known that   there is a homeomorphism  $\rho : (\Sigma, x_0)\to (\Sigma, x_0)$ such that $\rho(\delta_i)= \delta_{\tau(i)}$  for all $1\le i\le n$   and 
 $$BS(\beta) \cong \Sigma \times [0,1] / (x, 0)\sim (\rho(x), 1)$$  
In the case of a cable pattern  the homeomorphism $\rho$ is a rotation about $x_0$  through an angle of $2\pi (m/n)$  for some integer $m$ such that $gdc(m,n)=1$. In this case we denote $BS(\beta)$ by $CS(n,m)$.

Denote the free generators of $\pi_1(\Sigma, x_0)$ corresponding to  the boundary  paths of the removed disks $ \delta_1, \ldots , \delta_n$  by $x_1,\ldots,x_n$ respectively.  Let $t \in  \pi_1  (BS(\beta), x_0)$ be the element represented by  the loop  $x_0\times I\subset BS(\beta)$.  Therefore,   
$$\pi_1(BS(\beta), x_0)=\pi_1(\Sigma, x_0)\rtimes \mathbb{Z}$$ 
where the  action of $\mathbb{Z}=\langle t \rangle$ on $\pi_1(\Sigma, x_0)$ is given by  
$$tx_it^{-1}=a_ix_{\tau(i)}a_i^{-1} \ \ \text{ for } \  \ 1\leq i\leq n$$ 
for some words $a_1,\ldots, a_{n}\in \pi_1(\Sigma, x_0)$. Observe that    in the case  of a cable space $a_i=1$ for all $1\le i\le n$. Any element of $\pi_1(BS(\beta), x_0)$ is therefore  uniquely  written in the form $w  t^z$ with $w\in \pi_1(\Sigma, x_0)\cong F(x_1, \ldots , x_n)$ and $z\in \mathbb{Z}$. 
Note  also that $\pi_1(\Sigma, x_0)=\langle\langle x_1 \rangle\rangle$ (normal closure)  as any two elements of $\{x_1,\ldots ,x_{n}\}$ are conjugate in $\pi_1(BS(\beta), x_0)$.    
 
Let $A$ denote $\pi_1(BS(\beta), x_0)$.  We say that a subgroup $U\leq A$ is \emph{meridional} if $U$ is generated by finitely many conjugates of $x_1$.    For example,   $\pi_1(\Sigma, x_0)\le A$ is meridional. We will need the following result from \cite{BDJW} which explains the behavior of the  meridional subgroups with respect to the peripheral subgroups $$P_V:=\pi_1(\partial V(\mathfrak{k}_0), x_0) \ \ \text{ and } \ \  C_V:=\pi_1(\partial V , x_0).$$ Observe that $C_V$ is  generated by $m_V, l_V$ and $P_V$ is  generated by $x_1, t^n$.
\begin{lemma}{\label{C1}}
Let   $U=\langle  g_1x_1g_1^{-1}, \ldots , g_kx_1g_k^{-1}\rangle $ with $k\geq 0$ and $g_1, \ldots , g_k \in A$. 

Then either  $U=\pi_1(\Sigma, x_0)$ (and in this case $k\ge n$)  or $U$ is freely generated by $h_1x_1h_1^{-1}, \ldots , h_m x_1 h_m^{-1}$ 
with ${m}\leq k$ and $h_1, \ldots , h_m\in A$   such that:
\begin{enumerate}
\item for any $g\in A$ one of the following holds:
\begin{enumerate}
\item  $gP_Vg^{-1}\cap U=\{1\}$.
\item  $gP_Vg^{-1}\cap U=g \langle x_1 \rangle g^{-1}$ and $gx_1g^{-1}$ is in $U$ conjugate to $h_lx_1h_l^{-1}$ for some $l\in\{1,\ldots,{m}\}$.
\end{enumerate}
 
\item for any $g\in A$ the subgroups  $gC_Vg^{-1}$ and $U $ intersect trivially.
\end{enumerate}
\end{lemma}

\begin{remark}
The previous lemma tells us  that  the minimal number of conjugates of $x_1$ needed to generate a given meridional subgroup $U\le A$ coincides with  the rank of $U$, i.e.~the minimal number of elements needed to generate $U$. 
\end{remark}


\subsection{The centralizer of the meridian} This subsection is devoted to show a well-known fact, see \cite[Theorem~2.5.1]{Friedl} for example,  that says that  almost no element in the  knot group commutes with the meridian.   
\begin{lemma}{\label{lemma:commeridian}}
Let $\mathfrak k$ be a non-trivial prime knot and $g\in G(\mathfrak k)$. Then $$gmg^{-1}=m \ \text{ iff } \ g\in P(\mathfrak k).$$ 
\end{lemma} 
\begin{proof}
It is proved in Lemma~3.1 of \cite{Weidmann} that the peripheral subgroup of $\mathfrak k$ is malnormal in $G(\mathfrak k)$, i.e.~$gP(\mathfrak k)g^{-1}\cap P(\mathfrak k)=1$ for all $g\in G(\mathfrak k)\setminus P(\mathfrak k)$,  unless $\mathfrak k$ is a torus knot or a cable knot or a composite knot. The last case does not occur since $\mathfrak k$ is assumed to be  prime. In the case of a torus knot  it is  shown  in Lemma~3.2 of \cite{Weidmann}  that  $gmg^{-1}\cap P(\mathfrak k) =1 $  for all $g\in G(\mathfrak k)\setminus P(\mathfrak k)$.

We therefore need to consider the case  where $\mathfrak k$ is a cable knot.  Thus $G(\mathfrak k)$ splits as $\pi_1(\mathbb A, v_0)$ where $\mathbb A$ is the graph of groups described  in the previous subsection.  Observe that  the meridian of $\mathfrak k$  is represented by the reduced $\mathbb A$-path $ x_1\in A_{v_0}=\pi_1(CS(n,m), x_0)$, that is, $m=[x_1]\in \pi_1(\mathbb A, v_0)$.   

Let $g\in G(\mathfrak k)=\pi_1(\mathbb A, v_0)$ such that $gmg^{-1}=m$. We can write $g=[p]$  where 
$$p=a_0 , e , a_1, \ldots , a_{2l-1} , e^{-1}, a_{2l}$$
is a reduced $\mathbb A$-path of length $2l\ge 0$. Then $gmg^{-1}=m$ implies that  the $\mathbb A$-paths 
$$  a_0 , e , a_1, \ldots , e^{-1}, a_{2l} \cdot x_1 \cdot a_{2l}^{-1} , e, \ldots, a_1^{-1} , e^{-1}, a_0^{-1} \ \ \text{ and } \   \  x_1$$   
 are equivalent, see \cite[p.612]{RW} or \cite[Definition~2.3]{KMW}. If $l\geq 1$ then we can apply a reduction to  $p x_1p^{-1}$. Since $p$ is  reduced we  conclude that  $a_{2l} x_1 a_{2l}^{-1}$ is conjugate in the free group  $F_n=\langle x_1, \ldots , x_n\rangle\leq A_{v_0}$ to  an element of $\alpha_e(A_e)=C_V$. A simple homology argument shows that this cannot occur.     Thus $l=0$ and so the equality $gmg^{-1}=m$ reduces to $a_0x_1a_0^{-1}=x_1$.  Write $a_0= ut^z$ where $u\in F_n$ and $z\in \mathbb Z$. Then $$a_0x_1a_0^{-1}= u x_{\tau^z(i)} u^{-1}.$$  Hence  $a_0x_1a_0^{-1}=  x_1$   iff $\tau^z(1)=1$.  The second equality holds only when  $z=nk$ for some $k$.  Since $ux_1u^{-1}=x_1$  in $F_n$  we   conclude that  $u=x_1^w$ for some $w\in \mathbb Z$. Therefore,  $g=x_1^w(t^{n})^k$ which shows that  $g\in P_V=P(\mathfrak k)$. \end{proof}

 
\subsection{Algebraically tame knots} Let $\mathfrak k\subseteq \mathrm{S}^3$ be a knot. We call a subgroup $U \le  G(\mathfrak k)$ \emph{meridional} if $U$ is generated by finitely many meridians of $\mathfrak k$. The minimal number of meridians needed to generated a meridional subgroup $U$, denoted by $w(U)$, is called the \emph{meridional rank of $U$}. Observe that the knot group  $G(\mathfrak k)$  is meridional and   $w(G(\mathfrak k))$ equals the meridional rank $w(\mathfrak k)$ of $\mathfrak k$.

\begin{definition}
A meridional subgroup $U=\langle  g_1 m g_1 , \ldots , g_r m g_r\rangle  \le G(\mathfrak k)$ of meridional rank $r:=w(U)$ is called \emph{tame} if the following hold:
\begin{enumerate}
\item  the meridians $g_img_i$ and $g_jmg_j$ ($1\le i\neq j\le r$) are not conjugate in $U$.

\item for any $g\in G(\mathfrak k)$  either  $U\cap gP(\mathfrak k)g^{-1} = 1$ or  $$U\cap gP(\mathfrak k)g^{-1} = g\langle m\rangle g^{-1}$$  and $gmg^{-1}$ is in $U$ conjugate to $g_img_i$ for some $1\le i\le r$. 
\end{enumerate}
\end{definition}
\begin{definition}
A non-trivial knot $\mathfrak k$ is called \emph{algebraically tame} if any meridional subgroup  of $G(\mathfrak k)$ that is generated by less than $b(\mathfrak k)$  meridians is tame.
\end{definition}

\begin{remark}
Observe that for any algebraically tame knot $\mathfrak k$ the equality $ w(\mathfrak k)=b(\mathfrak k)$ holds since otherwise  $G(\mathfrak k)$ would be a tame subgroup. 
\end{remark}

\begin{remark}
A  (possibly) larger class of knots  (called meridionally tame knots) is  defined in \cite{BDJW}. The author does not know any example of a prime knot  that  is  algebraically tame but not meridionally tame.  
\end{remark}

\begin{example}
Whitehead doubles are never algebraically tame since their exterior contains  a properly immersed  $\pi_1$-injective  pair of pants in which two boundary components are mapped to meridional curves of  $\mathfrak k$, see~\cite{Agol}.
\end{example}

\begin{proposition}
Two bridge knots, torus knots and prime three bridge knots are algebraically tame. 
\end{proposition}
\begin{proof}
First assume that $\mathfrak k$ is a $2$-bridge knot. Thus we need to show that any  meridional subgroup generated by a single meridian is tame.   Let $U=\langle a ma^{-1}\rangle\le G(\mathfrak k)$. We can assume that $a=1$.  Since $\mathfrak k$ is neither a cable knot nor a composite knot, Lemma~\ref{lemma:commeridian} implies  that $gP(\mathfrak k) g^{-1} \cap U\neq 1$ iff $g\in P(\mathfrak k)$. Therefore $$gP(\mathfrak k) g^{-1} \cap U\neq 1 \  \ \text{ iff} \ \   g P(\mathfrak k) g^{-1} \cap U= \langle m\rangle.$$ Clearly  $gmg^{-1}$ is  conjugate in $U$ to $m$. 
 
Assume now that $\mathfrak k$ is a torus knot. It is shown in Theorem~1.2 of \cite{RostZ} that any meridional subgroup of $\mathfrak k$ of meridional rank $<b(\mathfrak k)$ is freely generated by meridians. This shows condition (1) of tameness. Condition (2) is implicit in \cite{RostZ} and an explicit argument is given in  \cite[Lemma~6.1]{BDJW}.

The algebraic tameness of prime $3$-bridge knots  follows from  the proof of  Proposition~7.1 of \cite{BDJW}.  
\end{proof}


\section{The pattern space of  a Whitehead double}
In this section we study the fundamental group  of the pattern  space $E:=V\setminus Int \ V(\mathfrak k_0)$   where $(V, \mathfrak k_0)$ is the Whitehead pattern and $V(\mathfrak k_0)$ is a regular neighborhood of $\mathfrak k_0$ in the interior of $V$.  
\begin{figure}[h!] 
\begin{center}
\includegraphics[scale=1]{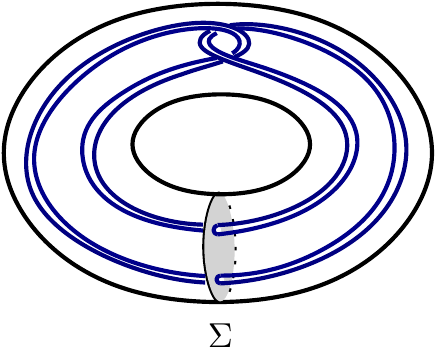}
\caption{The pattern space $E=V\setminus   Int  \ V(\mathfrak k_0)$.}{\label{fig:whiteheadspace}}
\end{center}
\end{figure}
Let $\Sigma\subseteq E$ be the properly embedded  pair of pants     shown in Figure~\ref{fig:whiteheadspace} and let $E_0$ be the space  obtained from $E$ by cutting along $\Sigma$, see Figure~\ref{fig:E0}. The group $\pi_1(E_0, e_0)$ (resp.~$\pi_1(\Sigma, \sigma_0)$)  is freely generated by the elements $x_1$ and $x_2$  (resp.~$y_1$ and $y_2$)  whose representatives are described in Figure~\ref{fig:E0}.   It follows from the theorem  of  Seifert and van-Kampen that   $\pi_1(E, e_0)$ splits as an HNN extension:
$$\pi_1(E, e_0) =\pi_1(E_0, e_0)\ast_{(\alpha, \pi_1(\Sigma, \sigma_0) , \omega)}$$ where:
\begin{enumerate}
\item[•] $\alpha :\pi_1(\Sigma, \sigma_0)\rightarrow \pi_1(E_0, e_0)$ is given by  $$\alpha(y_1)= x_2 x_1^{-1} x_2^{-1} \ \text{ and } \   \alpha(y_2)=x_1.$$  

\item[•]    $\omega:\pi_1(\Sigma, \sigma_0)\rightarrow \pi_1(E_0, e_0)$ is given by  $$\omega(y_1)=x_2x_1^{-1} x_2^{-1} x_1 x_2^{-1}  \ \text{ and } \  \omega(y_2)= x_2.$$  
\end{enumerate}
\begin{figure}[h!] 
\begin{center}
\includegraphics[scale=1]{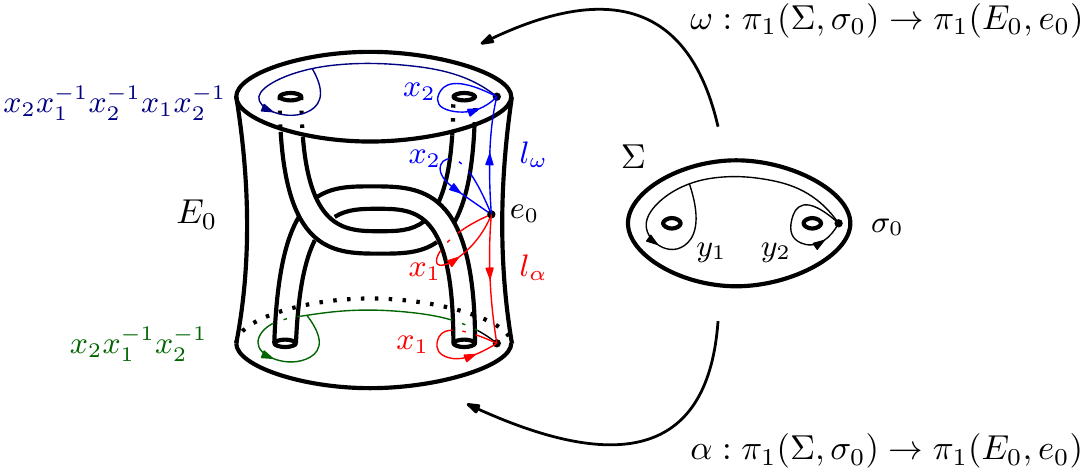}
\caption{The fundamental groups of $E_0$ and $\Sigma$.}{\label{fig:E0}}
\end{center}
\end{figure}

A presentation for $\pi_1(E, e_0)$ is therefore given by
$$\pi_1(E, e_0)= \langle x_1, x_2, l_V \ | \   l_V   x_2x_1^{-1} x_2^{-1} x_1 x_2^{-1}   l_V^{-1}=x_2 x_1^{-1} x_2^{-1} \ \text{ and  } \ l_V   x_2   l_V^{-1}=x_1\rangle$$
where $l_V$ is represented by the path $l_{\alpha}l_{\omega}^{-1}$ (which is  a   longitude of $V$). The meridian of $V$ (which coincides with the outer boundary of $\Sigma$)  represents the element $$m_V:=x_2x_1^{-1}x_2^{-1}  x_1= \alpha(y_1y_2)= \omega(y_1y_2)=x_2x_1^{-1} x_2^{-1} x_1 x_2^{-1} x_2.$$  
Define $C_V:=\langle m_V, l_V\rangle \cong \pi_1(\partial V, e_0)\leq \pi_1(\partial E, e_0)$.

\smallskip

For the next lemmas we denote  the subgroups $\alpha(\pi_1(\Sigma, \sigma_0)) $ and $ \omega(\pi_1(\Sigma, \sigma_0))$ of $\pi_1(E_0, e_0)$ by $U_\alpha$ and $U_\omega$ respectively.  
\begin{lemma}{\label{lemma:conjsep}}
$U_\alpha$ and $U_\omega$ are conjugacy separable in $\pi_1(E_0, e_0)$ meaning that  the intersection $U_\alpha\cap g U_\omega g^{-1}$ is trivial for all $g\in \pi_1(E_0, e_0)$.
\end{lemma}
\begin{proof}
The lemma follows by a simple homology argument.
\end{proof}

\begin{lemma}{\label{lemma:normalizer}}
Let  $a\in \pi_1(E_0, e_0)$. Then the following hold:
\begin{enumerate}
\item[(1)] $a\langle m_V\rangle a^{-1}\cap U_{\alpha}\neq 1$  iff $a\in U_{\alpha} $. Similarly, $a\langle m_V\rangle a^{-1}\cap U_{\omega}\neq 1$  iff $a\in U_{\omega} $

\item[(2)] $a\langle x_1\rangle a^{-1} \cap U_{\alpha}\neq 1$ iff  $a=ux_2^\varepsilon x_1^k$  with $u\in U_{\alpha}$,  $\varepsilon \in \{0,1\}$ and   $k\in \mathbb Z$. 

\noindent Similarly, $a\langle x_2\rangle a^{-1} \cap U_{\omega}\neq 1$ iff $a=u(x_2x_1^{-1})^\varepsilon x_2^k$  with  $u\in U_{\omega}$,  $\varepsilon \in \{0,1\}$ and   $k\in \mathbb Z$. 

\item[(3)] $aU_{\alpha}a^{-1} \cap U_{\alpha}\neq 1$ iff $a=u_1x_2^{\varepsilon}u_2$ with  $u_1, u_2\in U_{\alpha}$ and  $\varepsilon\in \{-1 ,0,1\}$.  Similarly,    $aU_{\omega}a^{-1} \cap U_{\omega}\neq 1$ iff $a=u_1(x_2x_1^{-1})^{\varepsilon}u_2$ with  $u_1, u_2\in U_{\omega}$ and  $\varepsilon\in \{-1 ,0,1\}$
\end{enumerate}
\end{lemma}
\begin{proof}
We   prove the above claims for $U_{\alpha}$. The argument for   $U_{\omega}$  is similar. We identify $\pi_1(E_0, e_0)$ with  the fundamental group of the rose $R_2$ with two petals. 

Let  $p:(\Gamma, u_0)\to (R_2, v)$ be the covering  corresponding to $U_{\alpha}= \langle x_1, x_2x_1^{-1}x_2^{-1}\rangle$, see   Figure~\ref{fig:cover}. 

\noindent (1) The result follows  from the fact a lifting $\widetilde{m}$ of $ m_V^k=(x_2  x_1^{-1} x_2^{-1} x_1)^k$ is closed iff $\widetilde{m}$ stars (and therefore also ends) at the vertex $u_0$.  
\begin{figure}[h!] 
\begin{center}
\includegraphics[scale=1]{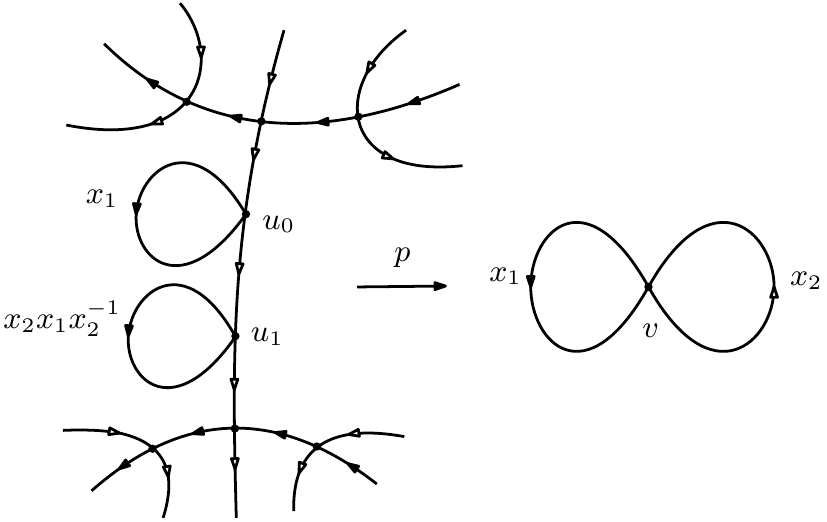}
\caption{The covering corresponding to $U_{\alpha}=\langle x_2x_1x_2^{-1}, x_1\rangle$.}{\label{fig:cover}}
\end{center}
\end{figure}

\noindent (2)  The argument is similar to the argument given in (1).  A lifting of $\tilde{l}$ of  $x_1^z$  is closed iff $\tilde{l}$ starts at $u_0$ or at $u_1$. Therefore any lifting of $a$  starting at $u_0$ either terminates at $u_0$ (which implies that $a$ lies in $U_{\alpha}$)  or  terminates at $u_1$ (which implies that $a=u x_2 x_1^k$ with $u\in U_\alpha$ and $k\in \mathbb Z$).

\noindent (3) First observe that reduced words (in the basis $\{x_1, x_2\}$ of $\pi_1(E_0, e_0)$) that represent elements of $U_\alpha$ are  of the form 
\begin{equation}\label{eq:normalform}
x_1^{n_0}\cdot x_2x_1^{n_1}x_2^{-1}\cdot   x_1^{n_2} \cdot \ldots \cdot x_1^{n_{d-1}}\cdot  x_2x_1^{n_d}x_2^{-1} \tag{*} 
\end{equation} 
with $n_0, n_d\in \mathbb Z$ and $n_1, \ldots , n_{d-1}\in \mathbb Z\setminus\{0\}$.

After multiplying $a$  on the right and on the left by elements of $U_\alpha$ we  can assume  that whenever  $a=u_1 a'u_2$  with $u_1, u_2\in U_\alpha$ and $|a'|\le |a|$ then $u_1=u_2=1$, where $|a|$ denotes the word length of $a$ with respect to the basis $\{x_1, x_2\}$. We will refer to this assumption as the \emph{length assumption on $a$}.

We will show that $a=x_2^{\varepsilon}$ for some $\varepsilon\in \{-1,0,1\}$.  
First observe that if $a=x_2^s$ with $s\in \mathbb Z$ then (\ref{eq:normalform}) implies that  $a U_\alpha a^{-1}\cap U_\alpha\neq 1$ iff $s\in \{-1,0,1\}$. 

Assume that $a$ is not a power of $x_2$. The length assumption on $a$ implies that $a$ can be  written as $x_2^r \cdot b\cdot  x_2^s$ with $r, s\in \mathbb Z\setminus\{0\}$ and  the first and last letters in $b$ are   $x_1^{\pm 1}$. If $s\neq -1$  then for any $u\in U$ the reduced word that represents $aua^{-1}$ contains the initial sub-word $x_2^r b$. From (\ref{eq:normalform})  we conclude that  that $r=1$.  But $b=x_1^{l}c$ for some $l\neq0 $ and some $c$ (which might be trivial). Thus $$a= x_2 x_1^l c x_2^s= x_2 x_1^l x_2^{-1} \cdot  x_2 c x_2^s.$$  with  $x_2x_1^lx_2^{-1}\in U$ and $|x_2 c x_2^s| <|a|$, a contradiction to the length assumption. Assume now that $s=-1$. The assumption on $b$ implies that  $b =cx_1^l$ for some $l\neq0$.  Thus $$a= x_2^r b x_2^{-1} =  x_2^r cx_1^l x_2^{-1}    =  x_2^r c x_2^{-1}\cdot  x_2x_1^l x_2^{-1} $$  with  $x_2 x_1^lx_2^{-1}\in U$ and $|x_2^r cx_2^{-1}|<|a|$  which  contradicts   the length assumption.\end{proof}

\begin{remark}{\label{rem:intersection}}
It follows from item (1) that    $a\langle m_V\rangle a^{-1}\cap U_\alpha$ is either trivial or equal to $a\langle m_V\rangle a^{-1}$. Item (2)  implies that    $a\langle x_1\rangle a^{-1} \cap U_{\alpha}$ is either trivial or equal to $a\langle x_1\rangle a^{-1}$. The same claims  also holds for $U_\omega$. 
\end{remark}

From Lemma~\ref{lemma:normalizer}(3)  we immediately obtain the following corollary.
\begin{corollary}{\label{cor:normalizer}}
 $ N_{\pi_1(E_0, e_0)}(U_\alpha)=U_\alpha $ and $ N_{\pi_1(E_0, e_0)}(U_\omega)=U_\omega$ 
where $N_{\pi_1(E_0, e_0)}(U )$ denotes the normalizer of $U\leq \pi_1(E_0, e_0)$. 
\end{corollary}

We will consider  subgroups of $\pi_1(E, e_0)$ that are generated by conjugates of $x_1$ and  $x_2$. We  call such subgroups \emph{meridional}.    We define the meridional rank $\bar{w}(U)$ of a meridional subgroup $U\le \pi_1(E, e_0)$ as the minimal number of conjugates of $x_1$ and $x_2$ needed to generate $U$.

\begin{definition}{\label{def:good}}
A meridional  subgroup $U = \langle g_1x_{i_1} g_1^{-1}, \ldots  , g_lx_{i_l}g_l^{-1}\rangle \le \pi_1(E, e_0)$ with $l=\bar{w}({U})$ is \emph{good} if there is a partitioned $J_1, \ldots , J_d$ of $\{1, \ldots  , l\}$ such that:
\begin{enumerate}
\item $U=U_1\ast \ldots \ast U_d$ where 
 $U_s$ ($1\le s\le l$) is generated by $\{g_k x_{i_k} g_k^{-1} \ | \  k\in J_s\}$.
\item for any $h\in \pi_1(E, e_0)$ one of the following holds:
\begin{enumerate}
\item $U \cap hC_Vh^{-1}=1$.

\item $U \cap  hC_Vh^{-1}=\langle h m_Vh^{-1}\rangle$ and there are   $1\le s\le d$  with  $|J_s|\ge 2$   and $u\in U$ such that $hm_Vh^{-1}$ lies in $ u  U _s u^{-1}$. 
\end{enumerate}  
\end{enumerate}
\end{definition}
 
The main result of this section  says that meridional subgroups are almost good in the sense that they are contained in a larger good subgroup whose rank does not increase.
\begin{proposition}{\label{prop:01}}
Any  meridional subgroup $ U \le\pi_1(E, e_0)$ is contained in a good meridional subgroup $\bar{U}\le\pi_1(E, e_0)$ with $\bar{w}(\bar{U})\le \bar{w}(U)$.
\end{proposition}

 
\subsection{Proof of Proposition~\ref{prop:01}} We consider the HNN extension $$\pi_1(E, e_0)=\pi_1(E_0, e_0)\ast_{(\alpha, \pi_1(\Sigma, \sigma_0), \omega)}$$  as a graph of groups $\mathbb A$ having a single vertex $v$ with group  $A_v=\pi_1(E_0, e_0)$ and a single edge pair $e, e^{-1}$ with $\alpha(e)=\omega(e)=v$ and group $A_e=A_{e^{-1}}=\pi_1(\Sigma, \sigma_0)$. The boundary monorphisms $\alpha_e, \omega_e:A_e\rightarrow A_v$  are given by $\alpha_e=\alpha$ and $\omega_e=\omega$.  
\begin{figure}[h!] 
\begin{center}
\includegraphics[scale=1]{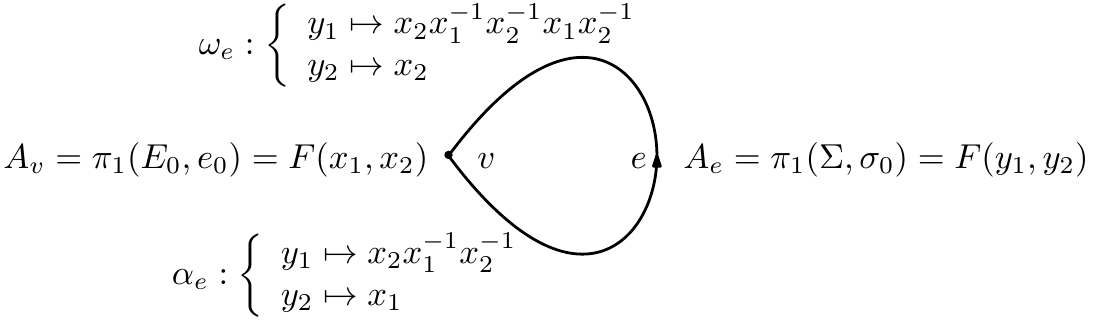}
\caption{The graph of groups $\mathbb A$.}{\label{fig:graphofgroupsWP}}
\end{center}
\end{figure}

The idea of the proof is to look at $\mathbb A$-graphs that  possibly represent good meridional subgroups of $\pi_1(E, e_0)\cong \pi_1(\mathbb A, v)$. We star by defining $\mathbb A$-graphs of  cyclic and non-cyclic type which will serve as the building blocks of the $\mathbb A$-graph we are going to consider.

\smallskip
 
We say that a folded $\mathbb A$-graph $\mathcal B$ is of \emph{cyclic type} if the following hold:
\begin{enumerate}
\item the underlying graph of $\mathcal B$ is an interval $f_1, \ldots , f_d$ of length $d\geq 0$.
 \begin{figure}[h!] 
\begin{center}
\includegraphics[scale=1]{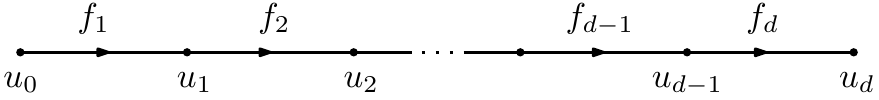}
\caption{The underlying graph of an $\mathbb A$-graph of cyclic type.}{\label{fig:graphcyclic}}
\end{center}
\end{figure}
\item for each $u\in VB=\{u_0, \ldots , u_d\}$ the corresponding  vertex group $B_u\leq A_v$ is generated by a single conjugate of $x_1$ or $x_2$.

\item for each $f\in EB=\{f_1^{\pm1 }, \ldots  , f_d^{\pm 1}\}$ the corresponding edge group $B_f\leq A_e$ is generated by a single conjugate of $y_1$ or $y_2$.
\end{enumerate} 

\smallskip
 
We say that a folded $\mathbb A$-graph  $\mathcal B$ is of \emph{non-cyclic type} if the following hold:
\begin{enumerate}
\item  the underlying graph $B$ of $\mathcal B$ contains an interval  $f_1, \ldots , f_d$ of length  $d\geq 2$ with $u_0=\alpha(f_1)$ and $u_i=\omega(f_i)$ for $i=1, \ldots , d$ such that $$B \setminus \{f_1^{\pm 1}, \ldots  , f_d^{\pm 1}\} = l_1\cup   l_2 \cup  l_3 \cup l_4$$ 
where $l_k$ ($1\le k\le 4$) is an interval  of length $d_k\ge 0$  with  $u_0=\alpha(l_1)=\alpha(l_2)$   and $u_d= \alpha(l_3)=\alpha(l_4)$, see Figure~\ref{fig:graphnoncyclic}.
\begin{figure}[h!] 
\begin{center}
\includegraphics[scale=1]{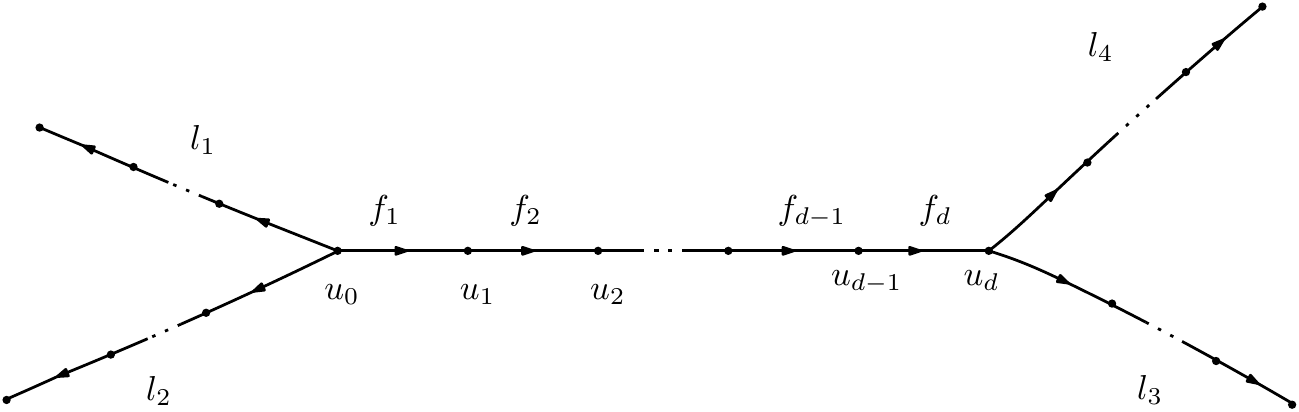}
\caption{The underlying graph of an $\mathbb A$-graph of non-cyclic type.}{\label{fig:graphnoncyclic}}
\end{center}
\end{figure}

\item for each $i\in \{1, 2, 3, 4\}$ the sub-$\mathbb A$-graph $\mathcal B_i$ of $\mathcal B$ that is carried  by the interval  $l_i$ is an $\mathbb A$-graph  of cyclic type. 

\item for each $1\le i\le d$, $B_{f_i}=A_{e_i}=A_e=\langle y_1, y_2\rangle$  where $e_i:=[f_i]\in EA$.

\item for each $1\le i\le d-1$,   $B_{u_i}=A_v=\langle x_1, x_2\rangle$. 
 
\item   the vertex groups at $u_0=\alpha(f_1)$ and $u_d=\omega(f_d)$ are given by: 
$$B_{u_0} = o_{f_1}\alpha_{e_1}(A_{e_1})o_{f_1}^{-1} =  o_{f_1} \langle \alpha_{e_1}(y_1) , \alpha_{e_1}(y_2)\rangle o_{f_1}^{-1}\le A_{v_0}$$  and $$B_{u_d}=t_{f_d}^{-1} \omega_{e_d}(A_{e_d}) t_{f_d} = t_{f_d}^{-1}\langle \omega_{e_d}(y_1), \omega_{e_d}(y_2)\rangle t_{f_d}\le A_{v_0}.$$
\end{enumerate}

We say that a connected sub-$\mathbb A$-graph $\mathcal B'$  of $\mathcal B$ is \emph{non-degenerate} if $\mathcal B'$  contains at least one vertex with $B_u=A_v$, or in other words,  if $\mathcal B'$  contains at least one of the vertices $u_1, \ldots , u_{d-1}$.

\smallskip
 
 Let $\mathcal B$ be an $\mathbb A$-graph of  cyclic or non-cyclic  type. It follows   from the definition  combined with \cite[Proposition~2.4]{KMW}  that the subgroup $U:=U(\mathcal B, u_0)\leq \pi_1(\mathbb A, v)$ represented by $\mathcal B$ is  meridional.  If $\mathcal B$ is of cyclic type,   then we further see that  $U(\mathcal B, u_0)$ is generated by a single conjugate of $x_1$ or $x_2$; hence $w(U)=1$. If $\mathcal B$ is  of non-cyclic type,   then is not hard to see that the sub-$\mathbb A$-graph $\mathcal B_0$  carried by $f_1, \ldots , f_d$ carries the fundamental group of $\mathcal B$, i.e.~the natural inclusion of $\mathcal B_0$ into $\mathcal B$  induces an isomorphism $\pi_1(\mathbb B_0, u_0)\to \pi_1(\mathbb B, u_0)$; hence  $U=U(\mathcal B_0, u_0)$. Moreover, to  $\mathcal B_0$ we associate a  manifold $X_{\mathcal B_0}$  whose fundamental group is canonically isomorphic to $\pi_1(\mathbb B_0, u_0)$ as shown in Figure~\ref{fig:essentialpiece1}.  
\begin{figure}[h!] 
\begin{center}
\includegraphics[scale=1]{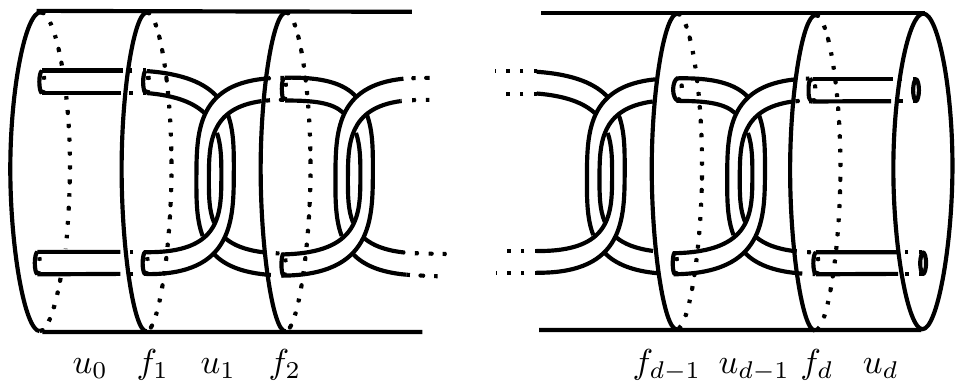}
\caption{The manifold $X_{\mathcal B_0}$.}{\label{fig:essentialpiece1}}
\end{center}
\end{figure}
Since $\mathcal B$ and therefore also $\mathcal B_0$ are folded  we conclude from \cite[Proposition~6]{RW} that $U=U(\mathcal B_0, u_0)$ is isomorphic to $\pi_1(\mathbb B, u_0)$ and hence  isomorphic to $ \pi_1(\mathbb B_0, u_0)\cong \pi_1(X_{\mathcal B_0})$. Thus $U$ is generated by $d$ conjugates of $x_1$ and $x_2$. Since the first homology group of $X_{\mathcal B_0}$ is isomorphic to $\mathbb Z^d$ we can conclude that   $w(U)=rank(U)=rank(\pi_1(X_{\mathcal B_0}))=d$.

\medskip

Let $\mathcal B$ be an arbitrary $\mathbb A$-graph. We define $\text{Ess}(\mathcal B)$ as the $\mathbb A$-graph that is obtained from $\mathcal B$ by removing all vertices and edges with non-trivial group.  Any component of $\text{Ess}(\mathcal B)$ is called an \emph{essential piece of $\mathcal B$}.
\begin{definition}
We say that an $\mathbb A$-graph $\mathcal B$ is \emph{good}  if the following hold:
\begin{enumerate}
\item[(G1)] the underlying graph of $\mathcal B$ is a finite tree. 

\item[(G2)] if $\mathcal B'$ is an essential piece of $\mathcal B$,  then  $\mathcal B'$  is  either an $\mathbb A$-graph of cyclic type (in this case we say that $\mathcal B'$ is a \emph{cyclic essential piece of $\mathcal B$}) or $\mathcal B'$ is a non-degenerate sub-$\mathbb A$-graph of  some $\mathbb A$-graph of non-cyclic type (in this case we say that $\mathcal B'$ is a \emph{non-cyclic essential piece of $\mathcal B$}).
\end{enumerate}
\end{definition}

The set of vertices and edges with non-trivial group  in a good $\mathbb A$-graph can be decomposed as follows. Let $\mathcal B$ be a good $\mathbb A$-graph. We define
$$V_\mathcal B:=\{u\in VB  \ | \ B_u\neq 1\} \  \ \text{ and } \ \ E_\mathcal{B}:=\{f\in EB \ | \ B_f\neq 1\}.$$
The set $V_\mathcal B$ is partitioned into three types of vertices, namely,  the set of: 
\begin{enumerate}
\item  \emph{cyclic vertices}  which consists of all vertices $u\in VB$  such that $B_u$ is generated by a single conjugate of $x_1$ or $x_2$.  

\item   \emph{almost full vertices} which consists of all vertices $u\in VB$ for which  there is  an edge $f\in Star_B(u)$ such that $B_f= \langle y_1, y_2\rangle$ and $$B_u = \langle o_f\alpha_{e'}(y_1)o_f^{-1} ,  o_f\alpha_{e'}(y_2)o_f^{-1}\rangle $$ where $e'=[f]\in EA$.  Observe that $B_{\omega(f)}=A_v$. 

\item  \emph{full vertices} which consists of all vertices $u$ such that $B_u=A_v=\langle x_1, x_2\rangle$.  
\end{enumerate}
Similarly, $E_\mathcal B$ is partitioned  into  two types of edges, namely,  the set of:

\begin{enumerate}
\item \emph{cyclic edges} which consists of all  edges $f\in EB$ such that $B_f$ is generated by a single conjugate of $y_1$ or $y_2$.

\item \emph{full edges} which consists of all  edges $f\in EB$ such that $B_f= \langle y_1, y_2\rangle$. 
\end{enumerate}
  
\smallskip

\begin{lemma}{\label{lemma:folds}}
Let $\mathcal B$ be a good $\mathbb A$-graph and let  $u \in VB$  and $f\in Star_B(u)$.  Assume that  $e'=[f]\in EA=\{e, e^{-1}\}$.  Then the following hold:
\begin{enumerate}
\item If $B_u$ is generated by a single conjugate of $x_1$ (in particular $u$ is cyclic) and $$o_f \alpha_{e'}(A_{e'}) o_f^{-1}\cap B_u \neq 1$$  then $e'=e$ and  up to auxiliary moves we can assume that  $B_u=\langle x_1\rangle$ and  $o_{f}= x_2^{-\varepsilon}$ with $\varepsilon \in \{0, 1\}$. 

\noindent Similarly, if $B_u$ is generated by a single conjugate of $x_2$ and 
$$o_f \alpha_{e'}(A_{e'}) o_f^{-1}\cap B_u \neq 1$$ then $e'=e^{-1}$ and up to auxiliary moves we can assume that $B_u=\langle x_2\rangle$ and $o_f= (x_2x_1^{-1})^{-\varepsilon }$ with $\varepsilon \in \{0,1\}$. 
 
\item If $B_u\neq A_v $ is generated by two conjugates of $x_1$ (in particular $u$ is almost full)  and 
$$ B_u \cap o_f\alpha_{e'}(A_{e'})o_f^{-1}\neq 1$$
then $e'=e$ and up to auxiliary moves $o_f=  x_2^{\varepsilon} $ with $\varepsilon \in \{-1, 0, 1\}$. 

\noindent Similarly, if $B_u\neq A_v$ is generated by two conjugates of $x_2$ (in particular $u$ is almost full) and 
$$ B_u \cap o_f\alpha_{e'}(A_{e'})o_f^{-1}\neq 1$$
then $e'=e^{-1}$ and up to auxiliary moves $o_f=  (x_2x_1^{-1})^{\eta} $ with $\eta \in \{-1, 0, 1\}$.
\end{enumerate} 
\end{lemma} 
\begin{proof}
Assume that $B_u=\langle g x_1 g^{-1}\rangle\leq A_v$ with $g\in A_v=\pi_1(E_0, e_0)$. Lemma~\ref{lemma:conjsep} implies that $e'=e$.  Lemma~\ref{lemma:normalizer}(2) implies that $$o_f^{-1} g =  \alpha_e(c) x_2^{\varepsilon} x_1^{k}$$ with $c\in A_e$, $\varepsilon\in \{0,1\}$  and $k\in \mathbb Z$.  Therefore $$o_f= g^{-1} x_1^k x_2^{-\varepsilon} \alpha_e(c).$$
This implies that after applying auxiliary moves to $\mathcal B$ based on $u$ and on $f$  we can assume that  $B_u=\langle x_1\rangle$ and $o_f=x_2^{-\varepsilon}$.   

\smallskip

\noindent (2) After applying an  auxiliary move of type A0 based on $u$   we can assume that $$B_u=\alpha_e(A_e)= \langle \alpha_e(y_1), \alpha_e(y_2)\rangle =  \langle  x_2x_1^{-1} x_2^{-1}, x_1\rangle\leq A_v.$$ 
By hypothesis $o_f \alpha_{e'}(A_{e'})o_f^{-1}\cap B_u\neq1$. Thus, by Lemma~\ref{lemma:conjsep}, it holds $e'=e$ and by Lemma~\ref{lemma:normalizer}(3) it holds $$o_f= u x_2^{\varepsilon } \alpha_e(c)$$ with $b\in B_u$ and $c\in A_e$. Therefore,  after applying  auxiliary moves  based on $u$ and of $f$, we can assume that $o_f=x_2^{\varepsilon}$.\end{proof}

\noindent\textbf{Collapsing essential pieces.} In addition to auxiliary moves as defined in~\cite[subsection~1.5]{RW}  we will also need another prepossessing move which is the inverse of a type IIA fold.

Let $\mathcal B$ be a good $\mathbb A$-graph. Let $\mathcal B'$  be an essential piece of $\mathcal B$  such that $\mathcal B'$ contains at least one edge (and hence $\mathcal B$ contains at least one edge with non-trivial group) and let  $u\in VB'$. Choose $u' \in VB'\setminus\{u\}$ such that: 
\begin{enumerate} 
\item[(i)] $val_{B' }(u')=1$ (in a finite tree there are at least two vertices with this property).

\item[(ii)] if $u''\in VB'\setminus\{u\}$ with  $val_{B'}{(u'')}=1$ then,  up to conjugacy,   $B_{u'}$ is a subgroup of $ B_{u''}$. 
\end{enumerate}
Let $f'\in EB'$ be the edge starting at $u'$. Denote  $u'':=\omega(f)\in VB'$ and $e'=[f']\in EA$.  We define a new $\mathbb A$-graph $\bar{\mathcal B}$ as follows: 
\begin{enumerate}
\item[1.] if $u'$ is not  full,  then  from  the definition  of $\mathbb A$-graphs of  cyclic and non-cyclic type  we see that 
 $B_{u'}=o_{f'} \alpha_{e'}(B_{f'})o_{f'}^{-1}$ .  In this case we replace the vertex group  $B_{u'}$ by $\bar{B}_{u'}=1$ and the edge group $B_{f'}\leq A_{e'}$ by $\bar{B}_{f'}=1$, see Figure~\ref{fig:collapse2}. 
\begin{figure}[h!] 
\begin{center}
\includegraphics[scale=1]{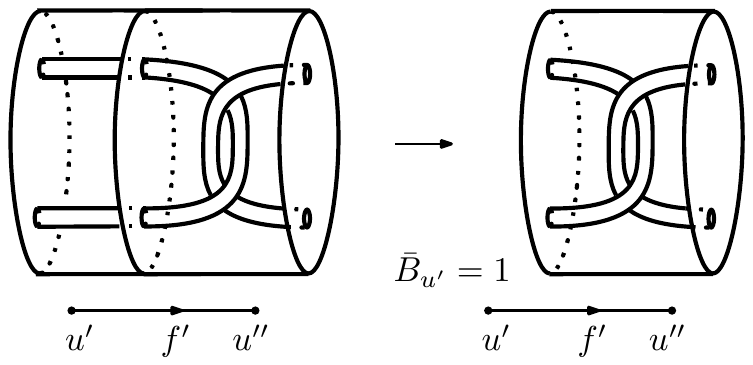}
\caption{Case 1: $u'$ is not full.}{\label{fig:collapse2}}
\end{center}
\end{figure} 

\item[2.]   if $u'$ is full,  then there are two cases:
\begin{enumerate}
\item[(a)] $u''$ is full or $u''$ is almost full and equal to $u$.
Observe that if the second option occurs,  then the choice of $u'$ guarantees that  $\mathcal B'$ contains only the edge $f'$.   As $u'$ is full we have
\begin{eqnarray}
B_u=A_v  & = &  \langle x_1, x_2\rangle   \nonumber\\
& = &  o_{f'} \langle x_1 ,  x_2 \rangle o_{f'}^{-1}\nonumber\\
& = &  \langle o_{f'}\alpha_{e'}(y_2)o_{f'}^{-1}, o_{f'} x_i o_{f'}^{-1} \rangle\nonumber
\end{eqnarray}   
where $i\in \{1, 2\}$ depends on  $e':=[f']$. In this case we do the following:
\begin{enumerate}
\item  replace the edge group $B_{f'}=A_{e'}=\langle y_1, y_2\rangle$ by $\bar{B}_{f'}=1$.

\item replace the vertex group  $B_{u'}=A_v$ by  $\bar{B}_{u'}= \langle o_{f'} x_i o_{f'}^{-1}\rangle$.

\item  If $u''$ is full then we do not change the vertex group at $u''$, i.e.~$\bar{B}_{u''}=B_{u''}$,   and if $u''$ is almost full then we replace  $$B_{u''}=\langle t_{f'}^{-1} \omega_{e'}(y_1) t_{f'} ,  t_{f'}^{-1} \omega_{e'}(y_2) t_{f'}\rangle$$ by $\bar{B}_{u''}=\langle t_{f'}^{-1} \omega_{e'}(y_2) t_{f'}\rangle$,  see Figure~\ref{fig:collapse1}.   
\begin{figure}[h!] 
\begin{center}
\includegraphics[scale=1]{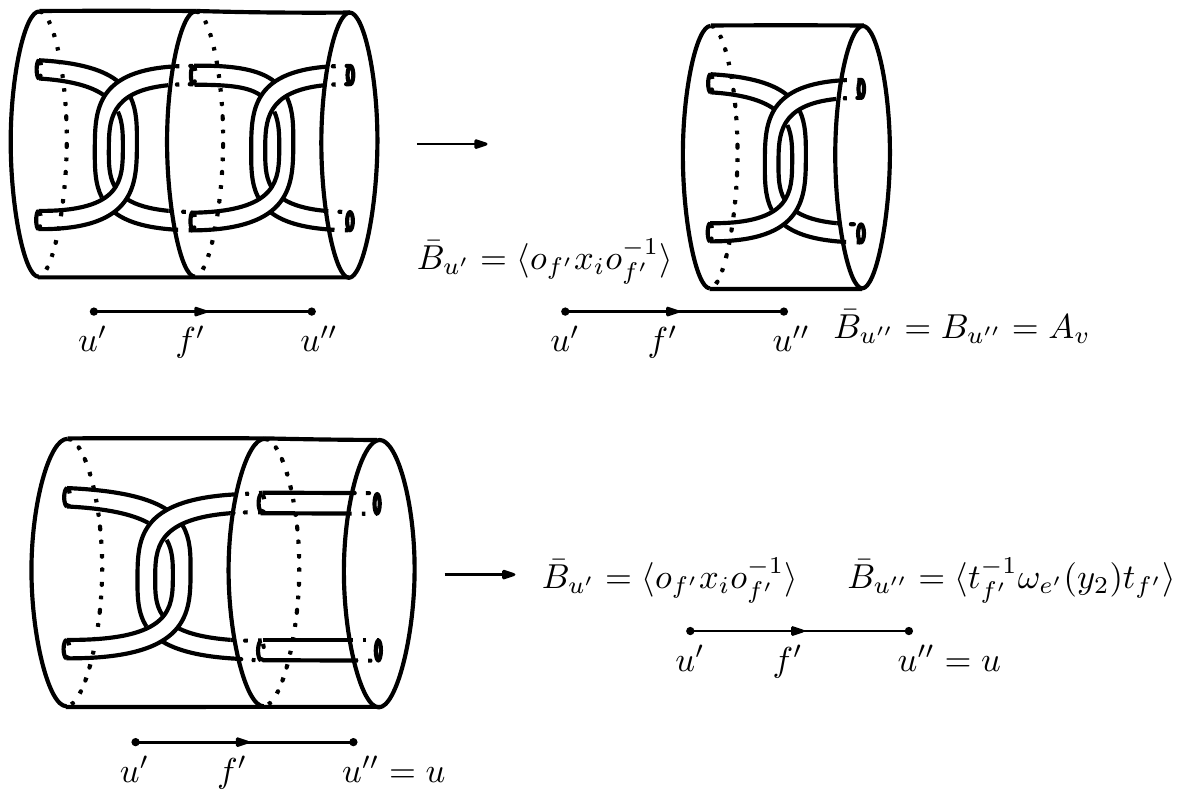}
\caption{Case (2.a): $u''$ is full or $u''$ is almost full and  $u''=u$.}{\label{fig:collapse1}}
\end{center}
\end{figure}
\end{enumerate}

\item[(b)] $u''$ is not full and distinct from $u$. The choice of $u'$ implies that $u$ is cyclic and $val_{B'}(u'')=2$. Let $f''$ be the edge in $Star_{B'}(u'')$ distinct from $(f')^{-1}$ and $e'':=[f'']$. The previous lemma implies that there is $h\in B_{u'}=A_v$ and $i,j\in \{1, 2\}$ such that 
$$B_{u'}=A_v=\langle o_{f'}\alpha_{e'}(y_j)o_{f'}^{-1} , hx_ih^{-1} \rangle$$  and $$ o_{f''}\alpha_{e''}(B_{f''}) o_{f''}^{-1} = \langle t_{f'}^{-1} \omega_{e'}(y_j) t_{f'}\rangle, $$
see Figure~\ref{fig:collapse3}. In this case we replace the edge group  $B_{f'}=A_e=\langle y_1, y_2\rangle$ by $\bar{B}_{f'}=1$ and  the vertex group   $B_{u'}$ by  $\bar{B}_{u'}= \langle hx_ih ^{-1}\rangle$.
\begin{figure}[h!] 
\begin{center}
\includegraphics[scale=1]{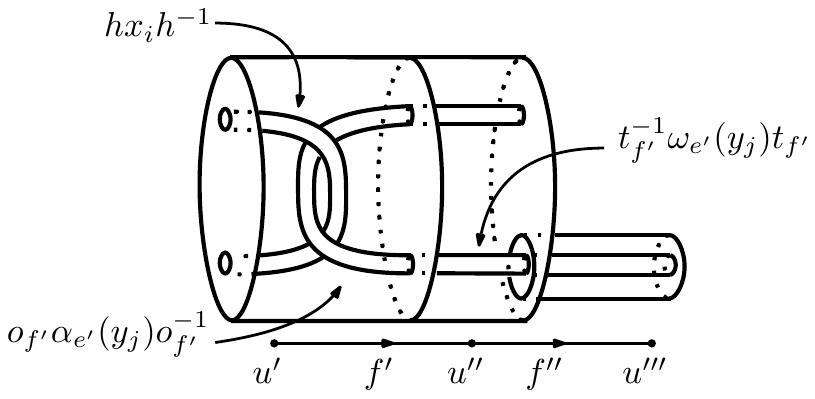}
\caption{Case (2.b):  $u''$ is not full and $u''\neq u$}{\label{fig:collapse3}}
\end{center}
\end{figure} 
\end{enumerate}
\end{enumerate}

It is not hard to see that  $\bar{\mathcal B}$ is indeed good.  We say that \emph{$\bar{\mathcal B}$ is obtained from $\mathcal B$ by unfolding along the edge $f$}. After repeating this procedure finitely many times we obtain a good $\mathbb A$-graph $\widetilde{\mathcal B}$  in which all edge groups in $\mathcal B'$ are trivial and $\widetilde{B}_u = B_u$ if $u$ is cyclic or full (in $\mathcal B$) and $\widetilde{B}_u$ is  generated by a  conjugate of $x_1$ or $x_2$ if $u$ is almost full (in $\mathcal B$). 
We say that \emph{$\widetilde{\mathcal B}$ is obtained from $\mathcal B$ by collapsing $\mathcal B'$ into $u$}.  It  is  clear from the definition that the  labels of  edges are not affected in these moves.  
   
\smallskip
 
Let $\mathcal B$ be a good $\mathbb A$-graph. The \emph{rank of $\mathcal B$}, denoted $rank(\mathcal B)$,   is defined as the rank of $\pi_1(\mathbb B, u_0)$. The number of full vertices (resp.~full edges)  in $\mathcal B$ is denoted by   $v_{full}(\mathcal B)$  (resp.~$e_{full}(\mathcal B)$). 
\begin{lemma}{\label{lemma:rank}}
Let $\mathcal B$ be a good $\mathbb A$-graph with base vertex  $u_0\in VB$. Let $\mathcal B_1, \ldots , \mathcal B_r$ (resp.~$\mathcal B_{r+1}, \ldots , \mathcal B_l$)  be the   cyclic (resp~non-cyclic) essential pieces  of $\mathcal B$.  Then
$$rank(\mathcal B)=  r + \sum_{i=r+1}^l (v_{full}(\mathcal B_i)+1)  $$
and the subgroup  $U(\mathcal B, u_0)\leq \pi_1(\mathbb A, v)=\pi_1(E, e_0)$ represented by $\mathcal B$ is generated by $rank(\mathcal B)$ conjugates of $x_1$ and $x_2$. 

If $\mathcal B$ is folded,  then $U(\mathcal B, u_0)$ is good and  $w(U(\mathcal B, u_0))=rank(\mathcal B)$.
\end{lemma}
\begin{proof}
For each $i=1, \ldots , l$ let $u_i$ be an arbitrary vertex of $\mathcal B_i$. Let further $p_i$ be a $\mathbb B$-path from $u_0$ to $u_i$ such that  its underlying path is reduced.   Since the group of any edge and the group of any vertex that is not in $\hbox{Ess}(\mathcal B)$ is trivial,  the fundamental group of $\mathbb B$ splits as 
$$\pi_1(\mathbb B , u_0) =  p_1 \pi_1(\mathbb B_1, u_1) p_1^{-1} \ast \ldots   \ast p_l\pi_1(\mathbb B_l, u_l) p_l^{-1},  $$
where $p_i \pi_1(\mathbb B_i, u_i) p_i^{-1}$ consists of all elements of $\pi_1(\mathbb B, u_0)$ represented by paths  of the form $p_i p p_i^{-1}$   where  $p$ is a closed path in $\mathbb B_i$ based at $u_i$.

Grushko's Theorem~\cite{Grushko}  implies that    $rank(\mathcal B)=\sum rank(\mathcal B_j)$.  As observed before the rank of any cyclic piece is   $1$ and the rank of any non-cyclic piece $\mathcal B_{r+i}$ is  $v_{full}(\mathcal B_i)+1$. Therefore we have the right formula for $rank(\mathcal B)$.

Assume now that $\mathcal B$ is folded.   It follows from  \cite[Proposition 6]{RW} that $U(\mathcal B, u_0)$ is isomorphic to $\pi_1(\mathbb B, u_0)$.  From this fact we readily conclude that $w(U(\mathcal B, u_0))=rank(\mathcal B)$ and that  
 $U(\mathcal B, u_0)=    U_1\ast \ldots \ast U_l$ 
where 
$$U_i:=\phi( p_i \pi_1(\mathbb B_i, u_i) p_i^{-1})\cong \pi_1(\mathbb B_i, u_i)   \ \text{ for } \  i=1, \ldots, l.$$ 
This  shows that condition (1) of the definition of good  subgroups of $\pi_1(E, e_0)$ is satisfied.

It remains to show that (2) also holds. Let $h\in \pi_1(E, e_0)=\pi_1(\mathbb A, v)$. Since the underlying graph of $\mathcal B$ is a tree  and $l_V$ is represented by the closed $\mathbb A$-path  $1,e,1$, we conclude that   $$U(\mathcal B, u_0) \cap hC_V h^{-1} \le h\langle m_V\rangle h^{-1}.$$  
Assume that $hm_V^zh^{-1}\in U(\mathcal B, u_0)$ for some $z\neq 0$.  The Normal Form Theorem~\cite[Proposition~2.4]{KMW} combined with the foldedness of  $\mathcal B$  implies that there must be  a vertex $u\in VB$ such that $B_u$ contains  $am_V^za^{-1}$ for some $a\in A_v$.  Since no cyclic subgroup of $A_v$ contains $m_V^z$ it follows  that $u$ is either almost full or full (in particular $u$  lies in a non-cyclic essential piece of $\mathcal B$). Thus $B_u$ also contains  $am_Va^{-1}$. This shows that for any $h\in \pi_1(E, e_0)$  one of conditions (2.a) or (2.b) from the definition of good subgroups holds. \end{proof}

\begin{proof}[proof of Proposition~\ref{prop:01}]  
Let $U=\langle g_1x_{i_1} g_1^{-1} , \ldots  , g_kx_{i_k}g_k^{-1}\rangle$ be a meridional subgroup of $\pi_1(E, e_0)$ of meridional rank $k:=w(U)$. Since $\pi_1(E, e_0)$ splits as $\pi_1(\mathbb A, v)$,   where $\mathbb A$ is the  graph of groups described above, each  $g_s$ ($1\leq s\leq k$)  can be written as $g_s = [p_s]$ where $p_s$ is a (non-necessarily reduced) closed $\mathbb A$-path of positive length. 

We define $\mathcal B_0$ as follows (see figure~\ref{fig:initial}). Start with a vertex  $u_0 $ and for  each $1\leq s\leq k$ we glue an interval $l_s$ subdivided into $length(p_s)$ segments to  $u_0$. The label of each $l_s$  is defined so that $\mu(l_s)=p_s$. The vertex group at $\omega(l_s)$ is $\langle x_{i_s}\rangle$ and the remaining vertex and edge groups are trivial.

Note that   $\mathcal B_0$ is good since all components of $\hbox{Ess}(\mathcal B_0)$ are (degenerate) $\mathbb A$-graphs of   cyclic type. Note further that   $\pi_1(\mathbb B_0, u_0)$ is freely generated by $$[l_1   x_{i_1}  l_1^{-1}], \ldots ,    [l_kx_{i_k}  l_k^{-1}]$$ and,   by definition,    $[l_s x_{i_s} l_s^{-1}]$ is mapped onto $g_s x_{i_s} g_ s^{-1}$. Therefore   $U(\mathcal B_0 , u_0)=U$ and  $k= w(U)=rank(\mathcal B_0)$. 
\begin{figure}[h!] 
\begin{center}
\includegraphics[scale=1]{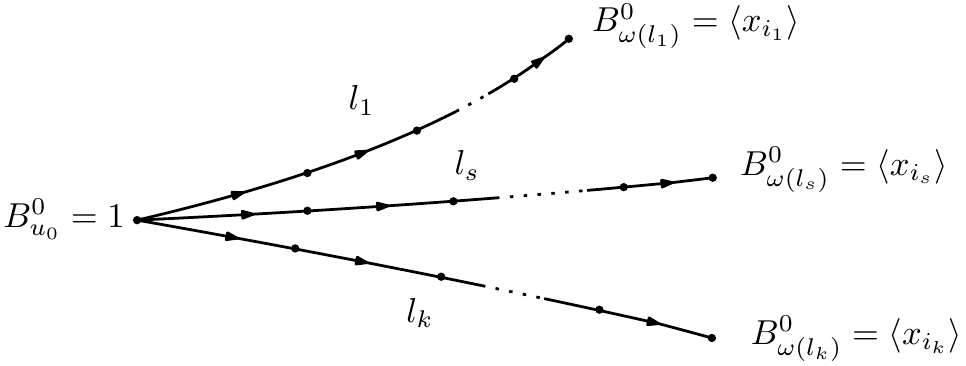}
\caption{The $\mathbb A$-graph $\mathcal B_0$.}{\label{fig:initial}}
\end{center}
\end{figure}

\begin{definition} 
We define the \emph{complexity} of a good $\mathbb A$-graph  $\mathcal B$  as    the tuple
$$c(\mathcal B) = (rank(\mathcal B), |EB|, |EB|-e_{full}(\mathcal B) , |EB|-|E_\mathcal B | ) \in \mathbb N_0^4.$$ 
Recall that $e_{full}(\mathcal B)$  denotes the number of full edges in $\mathcal B$ and $E_{\mathcal B}$ denotes  the set of edges in $\mathcal B$ with non-trivial group.  Throughout the proof we  assume that  $\mathbb N_0^4$ is equipped with  the lexicographic order.
\end{definition}
A straightforward inspection of the various cases reveals that all auxiliary moves preserve both goodness and  the  complexity, i.e.~if $\mathcal B'$ is obtained from a good $\mathbb A$-graph by an auxiliary move  then $\mathcal B'$ is also good and $c(\mathcal B')=c(\mathcal B)$. 

Now choose a good $\mathbb A$-graph  $\mathcal B$  such that the following hold:
\begin{enumerate}
\item[(1)]  $U\leq U(\mathcal B, u_0)$ for some  vertex $u_0$ of $\mathcal B$.

\item[(2)] the complexity of $\mathcal B$ is minimal among all good $\mathbb A$-graphs satisfying condition (1). 
\end{enumerate}

Since $\mathcal B_0$ is good and $U(\mathcal B_0, u_0)=U$  we conclude that   $rank(\mathcal B)\le rank(\mathcal B_0)=k=w(U)$. Thus $w(U(\mathcal B, u_0)) \le w(U)$. It remains to show that $U(\mathcal B, u_0)$ is good.     
According to Lemma~\ref{lemma:rank}  it suffices to show that $\mathcal B$ is folded.  We will argue by contradiction. Thus assume that $\mathcal B$ is not folded. Since the underlying graph of $\mathcal B$ is a tree, only folds of type IA and folds of type IIA can be applied to $\mathcal B$. 

\smallskip
 
\noindent\textbf{(IA)} Assume that a fold of  type IA can be applied to $\mathcal B$. Thus there are edges $f_1$ and $ f_2 $ in $\mathcal B$  with $u:=\alpha(f_1)=\alpha(f_2)$  and   $e'=[f_1]=[f_2]\in EA$ such that $$o_{f_2}=b o_{f_1} \alpha_{e'}(c)$$ for some $b\in B_u$ and some $c\in A_{e'}$.    After applying auxiliary moves to  $\mathcal B$  (see~\cite[section~1.5]{RW}) we can assume that the fold is elementary, that is, $o_{f_2}=o_{f_1}$. 

Let $\mathcal B'$ be any $\mathbb A$-graph that is obtained from $\mathcal B$ by collapsing  the essential pieces containing $\omega(f_1)$ and $\omega(f_2)$; hence,  there are no edges with non-trivial group in $\mathcal B'$ starting at $w(f_1)$ or $\omega(f_2)$.  Since the label of $f_1$ and $f_2$ is not affected, the elementary fold that identifies $f_1$ and $f_2$ can also be applied to $\mathcal B'$. 

Let $\mathcal B''$ be the $\mathbb A$-graph obtained from $\mathcal B'$ by this fold and let $y$ be the image of $\omega(f_1)$ (hence  also of $\omega(f_2)$) under the fold.  According \cite[Proposition~7]{RW} we have $U(\mathcal B'' , u_0'')=U(\mathcal B, u_0)$.  Since the vertex groups $B_{\omega(f_1)}'$ and  $B_{\omega(f_2)}'$  are replaced by   $$B_y''=\langle  B_{\omega(f_1)}', B_{\omega(f_2)}'\rangle \leq A_{[y]}$$ we see that  $B_y''$ is generated by at most four conjugates of $x_1$ and $x_2$.  If $B_y''$ is not cyclic, then we  replace $B_y''$ by $B_y'''=A_v=\langle x_1, x_2\rangle$.   In this way we obtain  a good $\mathbb A$-graph $\mathcal B'''$ such that $$U\le  U(\mathcal B'', u_0)\le U(\mathcal B''', u_0''').$$   The complexity decreases since  $rank(\mathcal B''')\le rank(\mathcal B)$ and  $|EB'''|= |EB|-2$.

\medskip
 
\noindent\textbf{(IIA)} Assume that no fold of type IA can be applied to $\mathcal B$. Since $\mathcal B$ is not folded   a fold of type IIA can be applied to $\mathcal B$. Thus  there is $f\in EB$    with $u:=\alpha(f)$,  $u':=\omega(f)$ and $e':=[f]\in EA=\{e, e^{-1}\}$    such that 
\begin{equation}{\label{eq:01}}
o_f \alpha_{e'}(B_f) o_f^{-1} \neq o_f\alpha_{e'}(A_{e'}) o_f^{-1} \cap B_u.
\end{equation}
Since all edges with non-trivial group  lie in some essential piece and these $\mathbb A$-graphs are folded  by definition,  we conclude that $B_f=1$.  Thus (\ref{eq:01}) means that 
\begin{equation}{\label{eq:02}}
o_f\alpha_{e'}(A_{e'}) o_f^{-1} \cap B_u \neq 1.
\end{equation}

To simplify the argument we assume that $e'=e$.  The case where $f$ is of type $e^{-1}$ is entirely analogous. We need to consider various cases depending on the vertex groups  $B_u$ and $B_{u'}$.

\smallskip

\noindent\textit{Case 1: $u$ is full, that is,  $B_u=A_v=\langle x_1, x_2\rangle$.}  The assumption that no fold of type IA can be applied to $\mathcal B$  implies that there is at most one edge in $Star_B(u)\setminus \{f\}$  which is necessarily of type $e^{-1}$.  We need to consider all possibilities for $B_{u'}$.   

\smallskip

\noindent \emph{(a) $B_{u'}$ is trivial.} We define  $\mathcal B'$ as  the $\mathbb A$-graph that is obtained from $\mathcal B$  by replacing  the edge group  $B_f=1$ by $B_f'=A_e=\langle y_1, y_2\rangle$ and the vertex group $B_{u'}=1$ by $$B_{u'}'=t_f^{-1} \omega_e(A_e) t_f= t_f^{-1} \langle \omega_e(y_1)  ,  \omega_e(y_2) \rangle t_f\le A_v,$$ see Figure~\ref{fig:IIAfull}.   
\begin{figure}[h!] 
\begin{center}
\includegraphics[scale=1]{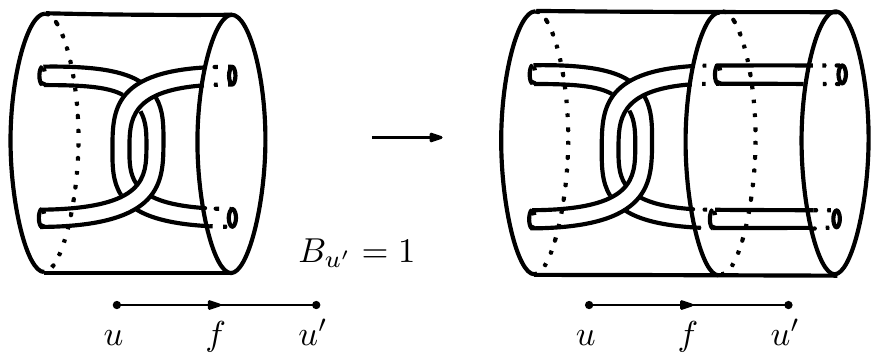}
\caption{Case (1.a): $B_{u'}'=t_f^{-1} \langle \omega_e(y_1)  ,  \omega_e(y_2) \rangle t_f$}{\label{fig:IIAfull}}
\end{center}
\end{figure}
Then $\mathcal B'$ is good since this move   simply ``enlarges'' the essential piece that contains $u$. Moreover,  $c(\mathcal B')<c(\mathcal B)$ since $rank(\mathcal B')=rank(\mathcal B)$, $|EB'|=|EB|$ and $e_{full}(\mathcal B')=e_{full}(\mathcal B)+2$.

\smallskip

\noindent\emph{(b) $u' $ is cyclic.} We first consider the case $Star_B(u')\cap E_{\mathcal B}$ is empty. We define $\mathcal B'$ as the $\mathbb A$-graph that is obtained from $\mathcal B$ by replacing the edge group $B_f=1$ by $B_f'=A_e=\langle y_1, y_2\rangle$ and the vertex group $B_{u'}$ by $B_{u'}'= A_v=\langle x_1, x_2\rangle$, see Figure~\ref{fig:IIAfull1}. It is not hard to see that $\mathcal B'$ is good. The complexity decreases  since  $rank(\mathcal B')=rank(\mathcal B)$, $EB'=EB$ and $e_{full}(\mathcal B')=e_{full}(\mathcal B)+2$.
\begin{figure}[h!] 
\begin{center}
\includegraphics[scale=1]{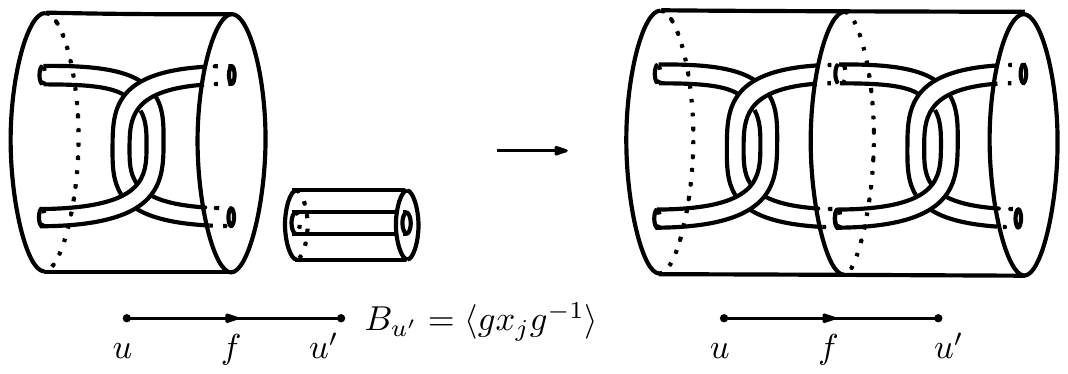}
\caption{Case (1.b). }{\label{fig:IIAfull1}}
\end{center}
\end{figure}

Assume now that $Star_B(u')\cap E_{\mathcal B}\neq \emptyset$. It follows immediately  from the definition of $\mathbb A$-graphs of (non-)cyclic  type that $ Star_{B}(u')\cap E_\mathcal B$ contains exactly one edge, say $h$, such that   $$B_h=\langle  cy_ic^{-1}\rangle \ \ \text{ and } \ \ B_{u'}=  o_h \langle \alpha_{e''}(cy_ic^{-1})\rangle o_h^{ -1}$$ for some $i\in \{1, 2\}$ and some $c\in A_{e}$ where $e'':=[h]\in EA$. Let $\mathcal B'$ be the $\mathbb A$-graph obtained from $\mathcal B$ by  doing the following:
\begin{enumerate}
\item[(a)] replace the edge group  $B_h=\langle c y_ic^{-1}\rangle$ by $B_h'=1$.

\item[(b)] replace the edge group  $B_f=1$ by $B_f'=A_e=\langle y_1, y_2\rangle$.  

\item[(c)] replace the vertex group  $B_{u'}= o_h\langle \alpha_{e''}(cy_ic^{-1})\rangle o_h^{-1}$ by 
$$B_{u'}= t_f^{-1} \omega_e(A_e) t_f = t_f^{-1} \langle   \omega_e(y_1) ,   \omega_e(y_2)\rangle t_f \le A_v.$$
\end{enumerate}  
\begin{figure}[h!] 
\begin{center}
\includegraphics[scale=1]{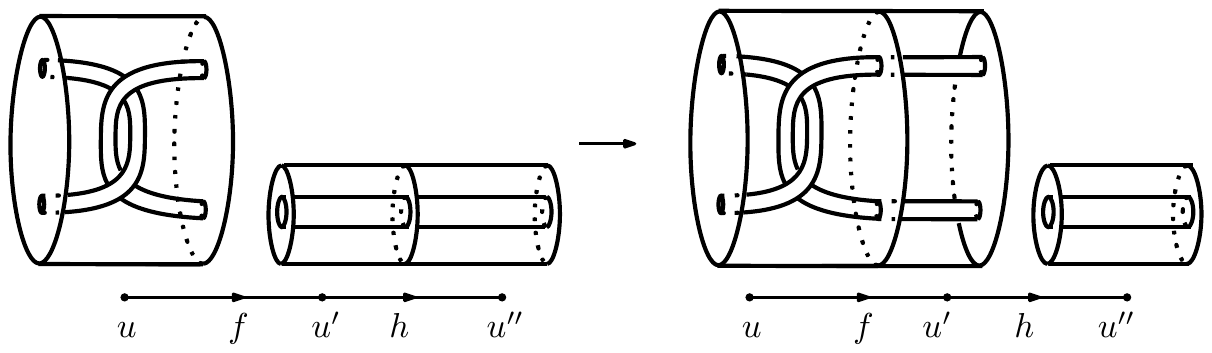}
\caption{Case (1.b).}{\label{fig:IIAfull5}}
\end{center}
\end{figure}
It is not hard to see that $\mathcal B'$ is good. Moreover,  $c(\mathcal B')<c(\mathcal B)$ since $rank(\mathcal B')= rank(\mathcal B)$, $|EB'|=|EB|$ and  $e_{full}(\mathcal B')=e_{full}(\mathcal B)+2$.

\smallskip

\noindent\emph{(c) $u'$ is almost full}. By definition    there is  an  unique  edge   $h\in EB$  such that  $B_h=A_{e''}=\langle y_1, y_2\rangle$, $B_{u''}=A_v=\langle x_1, x_2\rangle$ and 
$$B_{u'}= o_h \alpha_{e''}(A_{e''}) o_h^{-1}=o_h \langle \alpha_{e''}(y_1) , \alpha_{e''}(y_2) \rangle o_h^{-1} $$
where $e'':=[h]\in EA$ and $u''=\omega(h)\in VB$.   
After unfolding  along edges with cyclic group in the essential piece that contains $u'$  we can assume that $Star_B(u')\cap E_{\mathcal B}= \{h\}$. 

If $e'' =e$ then we define $\mathcal B'$ as the $\mathbb A$-graph that is obtained from $\mathcal B$ by replacing the edge group  $B_f=1$ by $B_f'=A_e$  and  the vertex group $B_{u'}$ by $B_{u'}'=A_v=\langle x_1, x_2\rangle$. Thus $\mathcal B'$ is good and $u'$ becomes a full vertex, see Figure~\ref{fig:IIAfull3}. 
\begin{figure}[h!] 
\begin{center}
\includegraphics[scale=1]{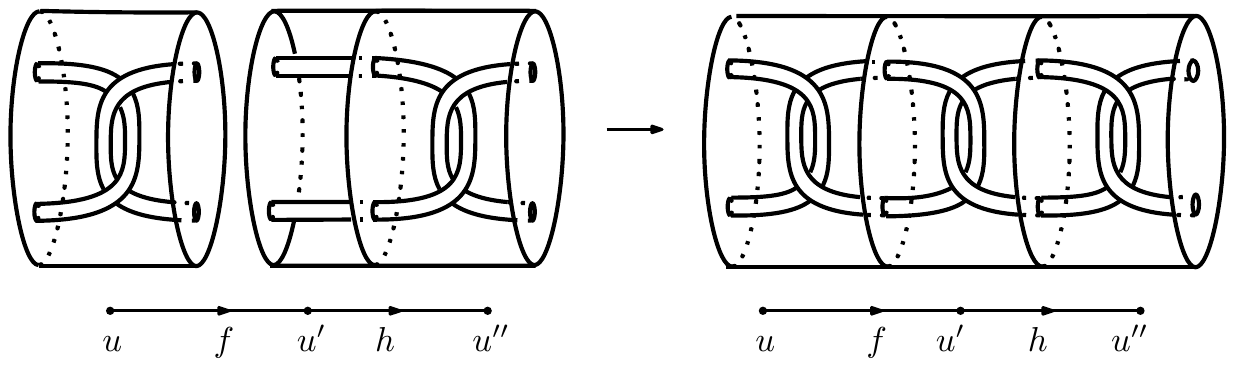}
\caption{Case (1.c).}{\label{fig:IIAfull3}}
\end{center}
\end{figure}
The complexity decreases since $rank(\mathcal B')=rank(\mathcal B)$, $EB'=EB$  and  $e_{full}(\mathcal B')=e_{full}(\mathcal B)+2$.

If $e''=e^{-1}$ then the previous move does not yield a good marked $\mathbb A$-graph since $f^{-1}$ and $h$ can then be folded and so cannot lie in an $\mathbb A$-graph of non-cyclic type. In this case we define a new  good $\mathbb A$-graph in which a fold of type IA can be applied to. 

First let $\mathcal B'$ be the $\mathbb A$-graph that is obtained from $\mathcal B$ by collapsing the  essential piece that contains $u$ into $u$. Thus $B_u'=B_u=\langle x_1, x_2\rangle$ since $u$  is full. After applying an auxiliary move of type A2 based on  the edge $f$  to $\mathcal B'$ we can assume that $o_f=1$. 

Now let $\mathcal B''$ be the $\mathbb A$-graph that is obtained from $\mathcal B'$  by pushing the element $x_1$ along $f$, that is, we replace $B_u'=\langle x_1, x_2\rangle$ by $B_u''=\langle x_2\rangle$ and $B_{u'}'=B_{u'}$ by $B_{u'}''=A_v$, see Figure~\ref{fig:IIAfull8}.\begin{figure}[h!] 
\begin{center}
\includegraphics[scale=1]{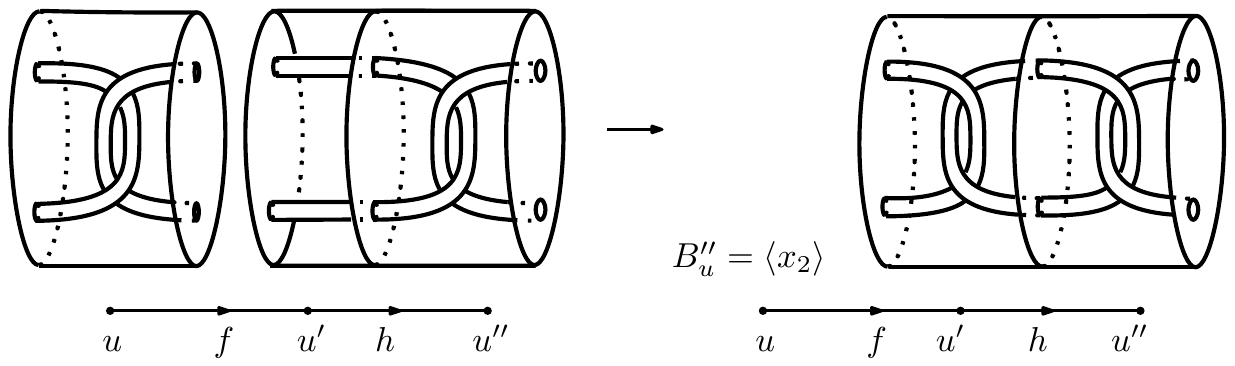}
\caption{Case (1.c).}{\label{fig:IIAfull8}}
\end{center}
\end{figure} 
The assumption that $Star_{B}(u')$ contains only $h$  ensures that  $\mathcal B''$ is good. Moreover, $rank(\mathcal B')=rank(\mathcal B)$. 

Since $f^{-1}$ and $h$  are of same type and $B_{u'}''=A_v$ it follows that a fold of type IA that identifies $f^{-1}$ with $h$ can be applied to $\mathcal B''$.  After folding  these edges we  obtain a good $\mathbb A$-graph $\mathcal B'''$.  To see that the rank decreases observe that the vertex groups $B_{u}''=\langle x_2\rangle$ and $B_{u''}''=A_v=\langle x_1, x_2\rangle$ are replaced by $A_v=\langle x_1, x_2\rangle$. Thus $c(\mathcal B''')<c(\mathcal B)$.

\noindent\emph{(d) $u'$ is full, that is, $B_{u'}=A_v=\langle x_1, x_2\rangle$.}  In this case we simply replace   $B_f=1$ by $B_f'= A_e=\langle y_1, y_2\rangle$, see Figure~\ref{fig:IIAfull2}.  The assumption that no fold of type IA  guarantees that  $\mathcal B'$ is good.      Moreover,   $rank(\mathcal B') \le rank(\mathcal B)-1$ which implies that $c(\mathcal B')<c(\mathcal B)$. 
\begin{figure}[h!] 
\begin{center}
\includegraphics[scale=1]{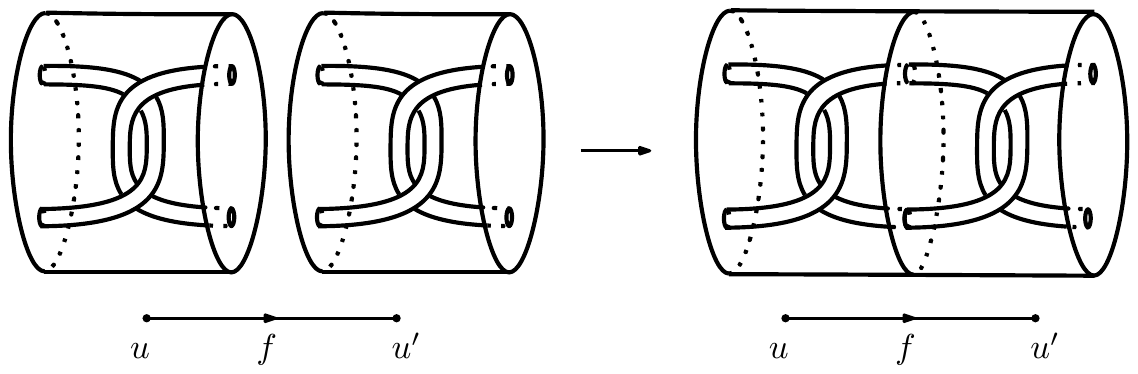}
\caption{Case (1.d).}{\label{fig:IIAfull2}}
\end{center}
\end{figure}

\medskip

\noindent\emph{Case 2: $u$ is almost full}. By definition,  there is a  unique  edge $h\in Star_B(u )$ such that $B_h=A_{e''}=\langle y_1 , y_2\rangle$, $B_{u''}=A_v=\langle x_1, x_2\rangle$ and  
$$B_{u } =o_h \alpha_{e'' } (A_{e''}) o_h^{-1}= o_h\langle \alpha_{e'' }(y_1), \alpha_{e''}(y_2)\rangle o_h^{-1}, $$  
where $e''=[h]\in EA$ and $u''=\omega(h)\in VB$. Since $$o_h\alpha_{e''}(A_{e''})o_h^{-1} \cap o_f \alpha_e(A_e)o_f^{-1} = B_u \cap o_f \alpha_e(A_e)o_f^{-1}\neq 1$$   it follows from Lemma~\ref{lemma:folds}(2) that  $e''=e$. Lemma~\ref{lemma:folds}(2) further implies that   $$B_u \cap o_f \alpha_e(A_e)o_f^{-1} = o_f \langle \alpha_e(y_i)\rangle o_f^{-1}$$ for some $i\in \{1, 2\}$. We can assume that $u'$ is not full since otherwise we apply the argument given in Case 1  to $f^{-1}$.  We distinguish two cases depending on the vertex group $B_{u'}$.

\smallskip

\noindent{\emph{(a)  $B_{u'}=1$.}} We define $\mathcal B'$  as the $\mathbb A$-graph obtained from  $\mathcal B$ by doing the following:
\begin{enumerate}
\item replace $B_f=1$ by $B_f'=\langle y_i\rangle$.

\item replace  $B_{u'}=1$ by $B_{u'}'= \langle t_f^{-1}\omega_e(y_i) t_f\rangle$.

\end{enumerate}
The construction of $\mathcal B'$ is described in  Figure~\ref{fig:IIAfull10}.  The assumption that no fold of type IA can be applied to $\mathcal B$  ensure that   $\mathcal B'$ is good. Moreover, the complexity decreases  as $rank(\mathcal B')=rank(\mathcal B)$, $|EB'|=|EB|$, $e_{full}(\mathcal B')=e_{full}(\mathcal B)$    and $|E_{\mathcal B'}|=|E_\mathcal B|+2$.
\begin{figure}[h!] 
\begin{center}
\includegraphics[scale=1]{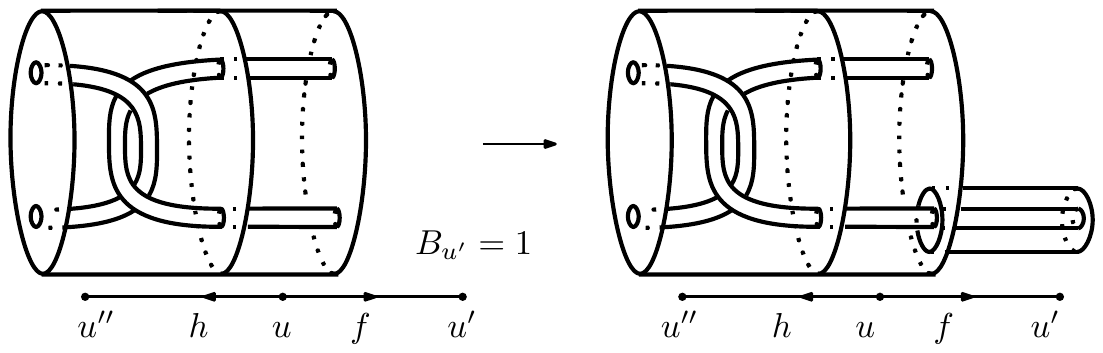}
\caption{Case (2.a).}{\label{fig:IIAfull10}}
\end{center}
\end{figure}

\smallskip

\noindent\emph{(b) $B_{u'}\neq 1$.} Let $\mathcal B(u)$ (resp.~$\mathcal B(u')$) denote the essential piece that contains $u$ (resp.~$u'$). We first construct $\mathcal B'$  by collapsing $\mathcal B(u')$  into $u'$ (hence $B_{u'}'\neq 1$) and by collapsing cyclic edges in $\mathcal B(u)$  so that $Star_{B'}(u)\cap E_{\mathcal B'}=\{h\}$.

Let now $\mathcal B''$ be the $\mathbb A$-graph that is obtained from $\mathcal B'$ by replacing the vertex group  $B_u'=B_u$ by $B_u''=A_v$, see Figure~\ref{fig:IIAfull511}. Note that the rank increases by one, hence  $c(\mathcal B'')>c(\mathcal B)$.  The edges $f$ and $h$ can be folded since  $u$ is full in $\mathcal B''$ and $f$ and $h$ are of same type. The $\mathbb A$-graph that is obtained by this fold is good and its rank is at most $rank(\mathcal B)$ since the vertices $u'$ and $u''$ are identified.  The complexity decreases because the number of edges decreases. 
\begin{figure}[h!] 
\begin{center}
\includegraphics[scale=1]{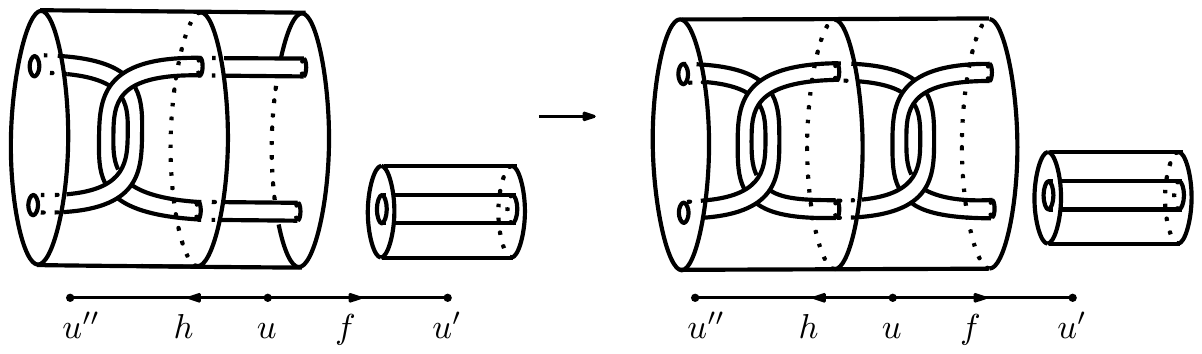}
\caption{Case (2.b). The construction of $\mathcal B''$. }{\label{fig:IIAfull511}}
\end{center}
\end{figure}

\medskip

\noindent\emph{Case 3: $u$ is cyclic.} As in the previous case we can assume that $u'$ is not full.  As $u$ is cyclic we have $B_u=\langle g x_i g^{-1}\rangle$ for some $g\in A_v$ and some $i\in \{1, 2\}$. Since the intersection   $B_u\cap o_f \alpha_e(A_e) o_f^{-1}$ is non-trivial, Lemma~\ref{lemma:conjsep} implies  that $i=1$.   Lemma~\ref{lemma:folds}(1)  tells us that  after auxiliary moves we have   $B_u=\langle x_1\rangle$  and $o_f= x_2^{-\varepsilon}$ for some $\varepsilon\in \{0,1\}$.     We distinguish two cases depending on the vertex group $B_{u'}$.

\smallskip

\noindent(a)  \emph{$B_{u'}=1$.} In this case we simply replace  the edge group $B_f=1$ by $B_f'=\langle y_i\rangle$ and the vertex group  $B_{u'}=1$ by $$B_{u'}'=\langle t_f^{-1} \omega_e(y_i)t_f\rangle  \leq A_v.$$ 
Observe that  $i=1$ if $\varepsilon =1$ and $i=2$ is $\varepsilon =0$.  The new $\mathbb A$-graph is clearly good. The complexity decreases because  $rank(\mathcal B')=rank(\mathcal B)$ , $|EB'|=|EB|$, $e_{full}(\mathcal B')=e_{full}(\mathcal B)$ and $|E_{\mathcal B'} |=|E_\mathcal B |+2$.

\smallskip

\noindent(b)  \emph{$B_{u'}\neq 1$.} Let $\mathcal B(u)$ (resp.~$\mathcal B(u')$) denote the essential piece of $\mathcal B$ that contains the vertex $u$ (resp.~the vertex $u'$).    We  need to consider various cases depending on $\mathcal B(u)$ and $\mathcal B(u')$.

\smallskip

\noindent\emph{$\mathcal B(u)$ and $\mathcal B(u')$ are cyclic essential pieces.} We first collapse  $\mathcal B(u)$ and $\mathcal B(u')$ into $u$ and $u'$ respectively. Then we  replace  the vertex group $B_{u'}$ by $B_{u'}'=A_v=\langle x_1, x_2\rangle$, the vertex group  $B_u $ by $\alpha_e(A_e)= \langle x_2 x_1^{-1} x_2^{-1}, x_1\rangle\leq A_v$ and the edge group  $B_f=1$ by $B_f'=A_e=\langle y_1, y_2\rangle$, see Figure~\ref{fig:IIAfull12}. This $\mathbb A$-graph is good and its complexity decreases since  the number of full edges increase by two.  Note that this move turns $u$ (resp.~$u'$) into an almost full vertex (resp.~full vertex).
\begin{figure}[h!] 
\begin{center}
\includegraphics[scale=1]{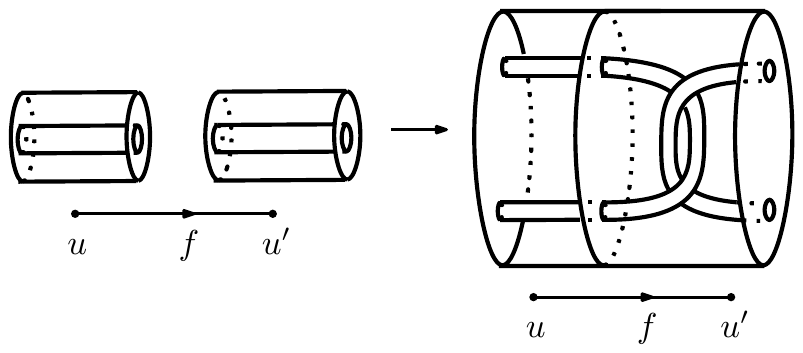}
\caption{Case (3).}{\label{fig:IIAfull12}}
\end{center}
\end{figure}

\smallskip

\noindent\emph{$\mathcal B(u)$ is a  non-cyclic essential piece. } After collapsing  $\mathcal B(u')$ into $u'$  we can assume that the vertex group  $B_{u'}$ is cyclic   and $Star_B(u')\cap E_{\mathcal B}=\emptyset$. Let $u''$ be the closest full vertex in $\mathcal B(u)$ to $u$ and let $f_1, \ldots , f_k$ be the shortest path from $u$ to $u''$.     After collapsing $\mathcal B(u)$  into $u$ we can assume that $B_u=\langle x_1\rangle$ and $B_{u''}=\langle g x_i g^{-1}\rangle$ for some $g\in A_v$ and some $i\in \{1, 2\}$. It follows from Lemma~\ref{lemma:folds} that $[f_i]=e$  if $i\in \{1, \ldots , k\}$ is odd and  $[f_i]=e^{-1}$ if $i\in \{1, \ldots , k\}$ is even.

Let $\mathcal B'$ be the $\mathbb A$-graph obtained from $\mathcal B$ by replacing the vertex groups $B_u$ and $B_{u'}$ by $B_{u}'=A_v=\langle x_1, x_2\rangle$ and  $B_{u'}'=A_v=\langle x_1, x_2\rangle$ and replacing the edge group  $B_f=1$ by $B_f'=A_e=\langle x_1, x_2\rangle$. Thus  $u$ and $u'$ (resp.~$f$) are full vertices (resp.~full edge) in $\mathcal B'$.  Observe further that $\mathcal B'$ is good  and represents a meridional subgroup that contains $U(\mathcal B, u_0)$ because we are replacing vertex groups and edge groups by larger ones. The only issue we need to solve is the rank  which increase  by one  in going from $\mathcal B$ to $\mathcal B'$. We remedy this as follows.  Since $u$ and $u'$ are full vertices and $[f_i]=e$  we conclude that the vertex  $u''=\omega(f_k)$ folds onto $u$ or $u'$ depending on the parity of $d$, see Figure~\ref{fig:IIAfull13}. Thus the rank of the resulting $\mathbb A$-graph clearly   drops back  to $rank(\mathcal B)$.  Since the number of edges decreases we see that $c(\mathcal B')<c(\mathcal B)$. 
\begin{figure}[h!] 
\begin{center}
\includegraphics[scale=1]{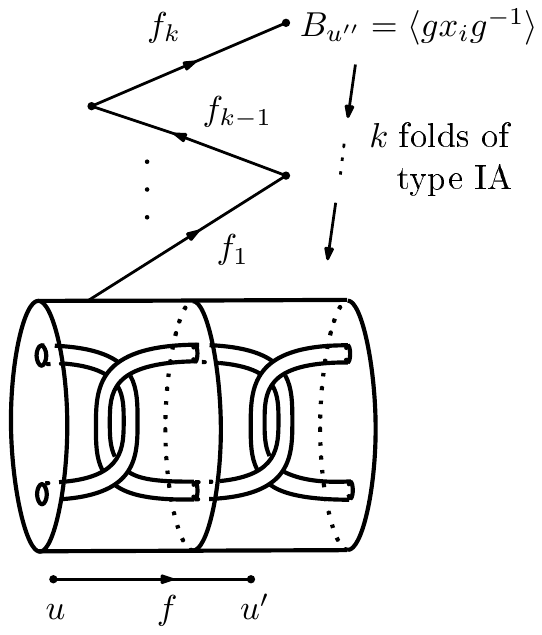}
\caption{Case (2.b). $u''$ folds onto $u$ or $u'$. }{\label{fig:IIAfull13}}
\end{center}
\end{figure}

\smallskip

\noindent\emph{$\mathcal B(u)$ is cyclic and $\mathcal B(u')$ is non-cyclic.} We modify $\mathcal B$ to fall in the previous case. After unfolding along some edges in $\mathcal B(u')$ we can  assume that $\{h\}= Star_B(u')\cap E_\mathcal B$.  Observe that $u'$  either cyclic or almost full.  

Let $\mathcal B'$ be the $\mathbb A$-graph  obtained from $\mathcal B$ by doing the following:
\begin{enumerate}
\item replace all vertex and edge groups in $\mathcal B(u)$ by trivial groups.

\item replace $B_{u'}$ by $B_{u'}'=\langle t_f^{-1} \omega_e(y_1) t_f\rangle$.

\item replace $B_{h}$ by $B_h=1$.
\end{enumerate}
Thus $\mathcal B'$ is good and the rank is not affected.  Now a fold along the edge $h^{-1}$ can be applied to $\mathcal B'$. To see that we fall in the previous case observe that the essential piece of $\mathcal B'$ that contains   $u'$ is cyclic  (as it contains only the vertex $u'$) and the essential piece of $\mathcal B'$ that contains $u'':=\omega(h)$ is non-cyclic since $u'$ is not full in $\mathcal B$.  \end{proof}


\section{Meridional rank of Whitehead doubles}
In this section we prove Theorem~\ref{thm01}.   We follow the notation introduced in Section~\ref{section:basicdef}.  We will first compute the bridge number of the Whitehead double of an arbitrary knot. Thus let $\mathfrak k_1$ be an arbitrary non-trivial knot  and let $\mathfrak k$ be a Whitehead double of $\mathfrak k_1$. We compute  $b(\mathfrak k)$ using the notion of plat presentations.  

We say that a knot $\mathfrak k'$ is a plat on $2n$-strings (or simply a $2n$-plat) if $\mathfrak k'$ is  the union of a  braid  with $2n$ strings  and $2n$ unlinked and unknotted arcs which connect pairs of consecutive strings of the braid at the top and at the bottom endpoints, see Figure~\ref{fig:plat}.  Any knot admits a plat presentation, that is,  it is ambient isotopic to a plat.  
According to \cite[Theorem~5.2]{Birman} the minimal $n$ such that $\mathfrak k'$ admits a plat presentation  with $2n$-strings coincides with the bridge number of $\mathfrak k'$.     
\begin{figure}[h!] 
\begin{center}
\includegraphics[scale=1]{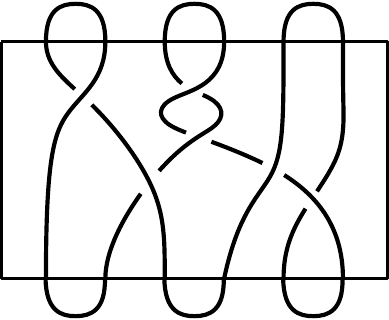}
\caption{A 6-plat.}{\label{fig:plat}}
\end{center}
\end{figure}

We  show that $b(\mathfrak k)=2 b(\mathfrak k_1)$.  A plat presentation with $2(2b(\mathfrak k_1))$-strings   for $\mathfrak k$ is described in Figure~\ref{fig:plat1}. Hence $b(\mathfrak k)\le 2b(\mathfrak k_1)$. On the other hand, as the index of  the Whitehead pattern $(V,\mathfrak k_0)$  is  $2$, it follows from \cite{Schu} (see also \cite{Schultens}) that $b(\mathfrak k)\ge 2b(\mathfrak k_1)$.    
\begin{figure}[h!] 
\begin{center}
\includegraphics[scale=1]{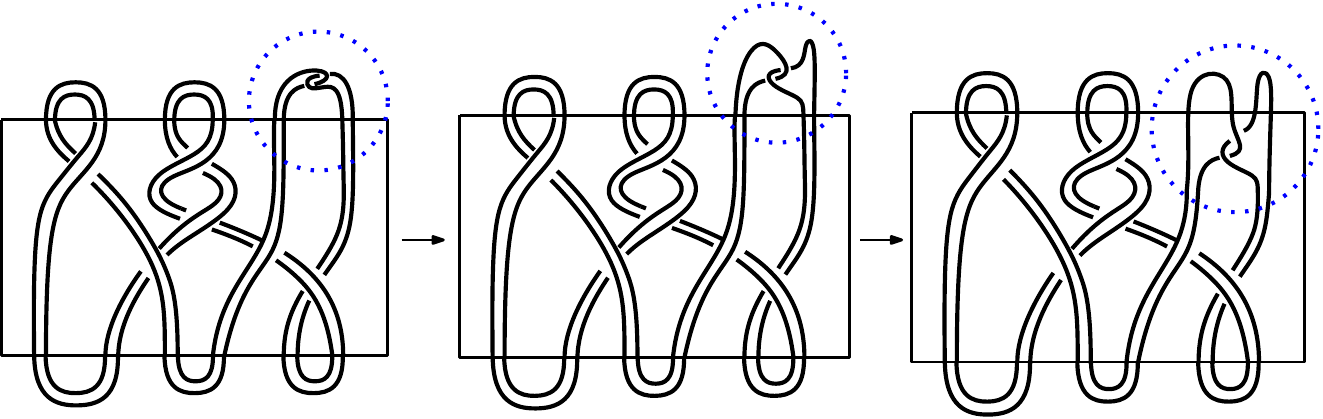}
\caption{Whitehead double of $\mathfrak k_1$.}{\label{fig:plat1}}
\end{center}
\end{figure}

\smallskip

We now show that  if $\mathfrak k_1$ is prime and algebraically tame then $w(\mathfrak k)=b(\mathfrak k)$.  Let $\mathbb A$ be the graph of groups   described in subsection~\ref{sub:sattelite}.   Recall that  
$$A_{v_0}=\pi_1(E, e_0)= \langle x_1, x_2, l_V \ | \   l_V \cdot  x_2x_1^{-1} x_2^{-1} x_1 x_2^{-1}\cdot   l_V^{-1}=x_2 x_1^{-1} x_2^{-1}  ,  l_V\cdot  x_2\cdot  l_V^{-1}=x_1\rangle. $$ 
By choosing the base point appropriately we  can assume that $m=[x_1]\in \pi_1(\mathbb A , v_0)$ is the fixed   meridian of $\mathfrak k$.

 We say that an $\mathbb A$-graph $\mathcal B$ is \emph{tame} if the following hold:
\begin{enumerate}
\item[(T0)] $\mathcal B$ is $\pi_1$-surjective meaning that the induced homomorphism  $$\phi:\pi_1(\mathbb B, u_0)\to \pi_1(\mathbb A , v_0)\cong G(\mathfrak k)$$ 
is surjective for some (and therefore) any vertex $u_0$ of type $v_0$.

\item[(T1)] the underlying graph $B$  is a finite tree. 

\item[(T3)] edge groups are either trivial or equal to $ \langle m_e\rangle$.     

\item[(T3)] for  each  vertex $u\in VB$  of type $v_0$  the  vertex group $B_u\leq A_{v_0}$ is good.

\item[(T4)] for each vertex $u\in VB$ of type $v_1$  the following hold:
\begin{enumerate}
\item  $B_u=\langle  g_1m_1g_1^{-1}, \ldots  , g_rm_1g_r^{-1}\rangle$  
where $  r= w(B_u) < b(\mathfrak k_1)$. 

\item for each $i=1, \ldots , r$ there is an edge $f_i\in   Star_B(u)\cap E_{\mathcal B}$ such that  $g_im_1g_i^{-1}$ is (in $B_u$) conjugate to $ o_{f_i} m_1 o_{f_i}^{-1}= o_{f_i} \alpha_e(m_e) o_{f_i}^{-1}$.
\end{enumerate} 
\end{enumerate}

The complexity of a tame $\mathbb A$-graph is defined  as the triple
$$c(\mathcal B):=(c_1(\mathcal B), |EB|, |EB|-|E_{\mathcal B}|) \in 	\mathbb N_0^3$$
where 
$$c_1(\mathcal B):= \sum_{\substack{ u\in VB  \\  [u]=v_0}}  \bar{w}(B_u)$$
Recall that $\bar{w}(B_u)$ denotes the minimal number of conjugates (in $A_{v_0}$) of $x_1$ and $x_2$  needed to generate $B_u$.  The set $\mathbb N_0^3$ is equipped with the  lexicographic order.

A straightforward inspection of the various cases reveals that all auxiliary moves preserve tameness and   complexity, i.e.~if $\mathcal B'$ is obtained from a tame $\mathbb A$-graph $\mathcal B$ by an auxiliary move  then $\mathcal B'$ is also tame and $c(\mathcal B')=c(\mathcal B)$.  
  
\begin{proof}[proof of Theorem~\ref{thm01}]
We want to show that $w(\mathfrak k)=b(\mathfrak k)$. The proof will be contradiction. Thus  assume that $$G(\mathfrak k)=\langle g_1mg_1^{-1} , \ldots  , g_lmg_l^{-1}\rangle$$ with $l<b(\mathfrak k)$. Each $g_i$ can be written as $g_i=[p_i]$ where $p_i$ is a non-necessarily reduced $\mathbb A$-path from $v_0$ to $v_0$ of  positive length.

We define an $\mathbb A$-graph $\mathcal B_0$ as follows. Start with a vertex  $u_0$ of type $v_0$ and for  each $1\leq s\leq l$ we glue an interval $l_s$ subdivided into $length(p_s)$ segments to $u_0$. The label of $l_s$  is defined so that  $\mu(l_s)=p_s$. The vertex group of $\omega(l_s)$ is $\langle  m\rangle$ and the  remaining vertex and edge groups are trivial. Therefore
$$\phi( [l_s m l_s^{-1}]) = [\mu(l_s) m \mu(l_s)^{-1}]=[ p_s mp_s^{-1}]= g_s x_{i_s} g_ s^{-1}, $$
and $\mathcal B_0$ is $\pi_1$-surjective. Observe that $\mathcal B_0$ is  tame since the only non-trivial vertex groups are cyclically generated by the meridian $m=[x_1]\in A_{v_0}$, and so are good subgroups of $A_{v_0}$.   

The claim that $w(\mathfrak k)<b(\mathfrak k)$ therefore  implies the existence  of a tame $\mathbb A$-graph.

\smallskip
 
Now choose a tame $\mathbb A$-graph $\mathcal B$ such that $c(\mathcal B)$ is minimal. Since $\mathcal B_0$ is tame it follows that   $c_1(\mathcal B)\le c_1(\mathcal B_0)=l<b(\mathfrak k)$.  
 
Since  all edge groups in $\mathcal B$ are proper subgroups of $A_{e}=\langle m_e, l_e\rangle$  we conclude that $\mathcal B$  is not folded. As the underlying graph of $\mathcal B$ is a tree we conclude that a fold of type IA or a fold  of type IIA can  be applied to  $\mathcal B$. We consider these  two cases separately.

\smallskip
 
\noindent\textbf{Fold of type IA.} In this case a pair of  distinct edges $f_1$ and $f_2$  starting at  a common vertex $u:=\alpha(f_1)=\alpha(f_2)\in VB$  can be folded. After   auxiliary moves   we can assume that  the fold is elementary, that is, $f_1$ and $f_2$ have  same label: $$(a, e', b):=(o_{f_1} , [f_1] , t_{f_1})=(o_{f_2} , [f_2] , t_{f_2})\in A_{[u]}\times EA \times A_{[\omega(f_1)]}.$$
Put $y_1:=\omega(f_1)$ and $y_2:=\omega(f_2)$. 

Let $\mathcal B'$ be the $\mathbb A$-graph that is obtained from $\mathcal B$  by folding $f_1$ and $f_2$.  By definition, the edge groups $B_{f_1}$ and $B_{f_2}$ are replaced by $$B_{f}'=\langle B_{f_1}, B_{f_2}\rangle\leq A_{e}$$  and the vertex groups  $B_{y_1}$ and $B_{y_2}$ are replaced by  $$B_y'=\langle B_{y_1}, B_{y_2}\rangle \le A_{[y]}, $$ see Figure~\ref{fig:elemfoldIA}. Observe that  $B_f'$ is either trivial or generated by $m_e$ and  that  $B_y' \leq A_{[y_1]}=A_{[y_2]}$ is meridional, that is, generated by conjugates of $x_1$ and $x_2$ if $[y_1]=v_0$ and generated by conjugates of $m_1$ if $[y_1]=v_1$.   
\begin{figure}[h!] 
\begin{center}
\includegraphics[scale=1]{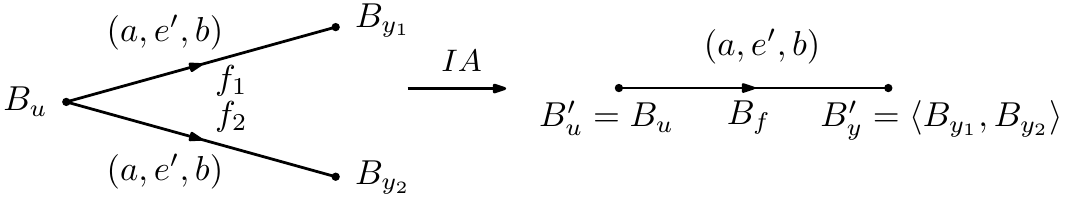}
\caption{An elementary fold of type IA.}{\label{fig:elemfoldIA}}
\end{center}
\end{figure}
We consider two cases depending on the type   $e'$ of $f_1$ and $f_2$.

\smallskip
 
\noindent\emph{Case 1:   $e'=e$}. In this case $u$ is of type $v_0$ and  $y_1  $ and $y_2$   are of type $v_1$. The only tameness condition that is not trivially satisfied is condition (T4.a).  Condition (T4.b) says that there are $w(B_{y_i})$ edges in $$S_i:=Star_B(y_i)\cap E_\mathcal B=\{f \in EB \ | \ B_f\neq 1 \ \text{ and } \ \alpha(f)=y_i\}.$$
Let $f\in S_1\cup S_2$. Since $B_f=\langle m_e\rangle$ it follows that   $t_f^{-1} m_V t_f\in B_{\omega(f)}$. Condition (2) in the definition of good subgroups of $A_{v_0}=\pi_1(E, e_0)$ then implies that $\bar{w}(B_{\omega(f)})\geq 2$. Therefore, if $\omega(f)\neq \omega(f')$ for all $f\in S_1$ and all $f'\in S_2$ (or equivalently, if $f_1^{-1}\cup f_2^{-1}$ is not contained in $S_1\cup S_2$) then  
\begin{eqnarray}
2w(B_y') & \le & 2w(B_{y_1})+2w(B_{y_2})  \nonumber \\
 & \le & 2|S_{1}|+2|S_{2}| \nonumber \\
  & \le & \sum_{f\in S_1\cup S_2} \bar{w}(B_{\omega(f)}) \nonumber\\
  &\le  &  c_1(\mathcal B)\nonumber 
\end{eqnarray}
and so  $2w(B_y')\le c_1(\mathcal B )\le l< b(\mathfrak k)=2b(\mathfrak k_1)$ which implies that $w(B_{y}')<b(\mathfrak k_1)$. If $f_1^{-1} \in S_1$ and $f_2^{-1}\in S_2$ then, since $B_{y_i}\le A_{v_1}$ is tame  and the meridian  $$t_{f_i}^{-1} m_1 t_{f_1} = b^{-1} m_1 b$$ lies in $B_{y_i}$, we can assume that $b^{-1} m_1 b$ is part of a minimal meridional generating set of $B_{y_i}$. Thus $$w(B_y')\le w(B_{y_1})+w(B_{y_2})-1.$$ The same computation as above  with the summation taken over the set $S_1\cup S_2\setminus\{f_1^{-1}\}$ shows that $w(B_y')< b(\mathfrak k_1)$. Therefore $\mathcal B$ is tame.  Since no vertex of type $v_0$ is affected we see that $c_1(\mathcal B')\le c_1(\mathcal B)$ and since the number of edges drops by two we conclude that $c(\mathcal B')<c(\mathcal B)$.

\smallskip
 
\noindent\emph{Case 2. } \emph{$e'=e^{-1}$}. Thus $u$ is of type $v_1$ and  $y_1$ and $y_2$ are of type $v_0$. It follows from Proposition~\ref{prop:01} that    $B_y'=\langle B_{y_1}, B_{y_2}\rangle \le A_{v_0}$ is contained in a good subgroup $B_{y}''\leq A_{v_0}$ such that  $\bar{w}(B_y'') \le \bar{w}(B_w')$. 

Let $\mathcal B''$ be the $\mathbb A$-graph  that is obtained from $\mathcal B'$ by replacing the vertex group  $B_{y}'$ by $B_{y}''\leq A_{v_0}$.  Thus $\mathcal B''$ is tame and 
 $$c_1(\mathcal B'')-c_1(\mathcal B)  = \bar{w}(B_y'') - \bar{w}(B_{y_1})-\bar{w}(B_{y_2}) \leq 0.$$ 
 Since two edges are identified we conclude that  $c(\mathcal B'')<c(\mathcal B)$ which contradicts   our choice of $\mathcal B$.

\medskip
 
\noindent\textbf{Fold of type IIA.} We can assume that no fold of type IA is applicable to $\mathcal B$.  By definition of type IIA folds,  there is  an edge  $f\in EB$  such that $$B_f\neq \alpha_{e'}^{-1}(o_f^{-1} B_x o_f)$$ where $e' :=[f]\in EA$ and $x:=\alpha(f)\in VB$.  Let $y$ denote the vertex $\omega(f)\in VB$. We consider two cases depending on the type $e'$ of $f$.

\smallskip

\noindent\emph{Case 1.} \emph{$e'=e^{-1}$}. Thus $x$ is of type $v_1$ and $B_x$ is a tame subgroup of $A_{v_1}= G(\mathfrak k_1)$.  Assume that
$$B_x= \langle h_1 m_1h_1^{-1} , \ldots  , h_rm_1h_r^{-1}\rangle$$
where $r=w(B_x)$. From  $$\alpha_{e^{-1}}^{-1}(o_f^{-1}B_xo_f)\neq 1$$ it follows that  
 $$  o_f\alpha_{e^{-1}}(A_e)o_f^{-1}\cap B_x \neq 1.$$
The tameness of the  meridional  subgroup $B_x\leq A_{v_1}$ implies  that 
$$  o_f\alpha_{e^{-1}}(A_e)o_f^{-1}\cap B_x= o_f \langle m_1\rangle o_f^{-1}$$
and  $o_fm_1o_f^{-1}$ is in $B_x$ conjugate to $h_im_1h_i^{-1}$ for some $1 \le i\le r$.  Condition (T4.b) says that  there is  $g\in Star_B(x)\cap E_{\mathcal B}$  such that $o_g m_1 o_g^{-1}$ is in $B_u$ conjugate to $ h_im_1h_i^{-1}$. Hence,  
$ o_g m_1 o_g^{-1}$ and $o_f m_1o_f^{-1}$  are  conjugate in $B_x$.   Lemma~\ref{lemma:commeridian} then implies that $o_f= a o_g c$ for some $a\in B_x$ and some $c \in P(\mathfrak k_1)$. This equality implies that a fold of type IA can be applied to $\mathcal B$ since  $P(\mathfrak k_1)=\alpha_{e^{-1}}(A_{e^{-1}})= \omega_e(A_e), $ which contradicts our assumption on $\mathcal B$.

\noindent\emph{Case 2.} \emph{$e'=e$.} Hence $x$ is of type $u_0$ and so  $B_x$ is a good subgroup of $A_{v_0}$. Thus 
$$ o_f \alpha_e(A_e) o_f^{-1} \cap B_u= o_f C_V o_f^{-1} \cap B_u= o_f \langle m_V \rangle o_f^{-1} $$
since $o_f \alpha_e(A_e) o_f^{-1} \cap B_u$ is non-trivial.

Let $\mathcal B'$ be the $\mathbb A$-graph that is obtained from $\mathcal B$  by replacing $B_y$ by 
$$B_y'=\langle B_y,  t_f^{-1} m_1 t_f\rangle$$ 
and $B_f=1$ by $B_f'=\langle m_e\rangle$. The only tameness condition that is not trivially satisfied is condition (T4.a), i.e.~$w(B_y')<b(\mathfrak k_1)$. To see this we apply the same computation as in the case of a fold of type IA based on an  edge of type $e$.  The complexity clearly decreases since no vertex of type $v_0$ is affected, $|EB'|=|EB|$ and $|E_{\mathcal B}|=|E_\mathcal B|+2$.      
\begin{figure}[h!] 
\begin{center}
\includegraphics[scale=1]{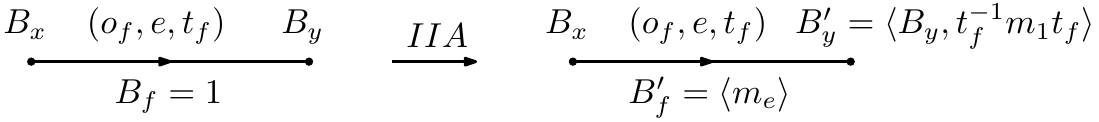}
\caption{An elementary fold of type IIA when  $[f]=e$.}{\label{fig:initial}}
\end{center}
\end{figure}
\end{proof}

  
\section{Algebraically tame knots and braid satellites}
In this section we prove Theorem~\ref{thm02}. We follow the notation from Section~\ref{section:basicdef}. Let $\mathfrak k_1$ be a prime  algebraically tame knot and let $\mathfrak k$ be a braid satellite of $\mathfrak k_1$ with   braid pattern $\beta$.  We want to show that $\mathfrak k$ is   algebraically tame. To this end we must show that any meridional subgroup of $G(\mathfrak k)$  that is generated by less than $b(\mathfrak k)$ meridians is tame where $n$ is the number of strand of $\beta$.  Let $\mathbb A$ be the graph of groups described in Section~\ref{section:basicdef}. Recall that
$$A_{v_0}= F(x_1, \ldots , x_n) \rtimes  \langle t \rangle $$
where the action of $\langle t\rangle\cong \mathbb Z$ in $F(x_1, \ldots x_n)$ is given by
$$tx_it^{-1}=a_i x_{\tau(i)} a_i^{-1} \ \ \text{ for } \ \ 1\le i\le n$$
where $a_1,\ldots, a_n\in F(x_1, \ldots , x_n)$ and $\tau\in S_n$ is the permutation associated to  $\beta$. By choosing the base point appropriately we can assume that $m=[x_1]\in \pi_1(\mathbb A , v_0)$ is the fixed meridian of $\mathfrak k$.

\medskip 

Thus  let $U\leq G(\mathfrak k)$ be a meridional subgroup with $r:=w(U)<b(\mathfrak k)$ (it follows from  \cite{Schu} that $b(\mathfrak k)=n b(\mathfrak k_1)$). In order to show that $U$ is tame  we study   $\mathbb A$-graphs that represent $U$. However, we do not consider arbitrary ones  but only those that are ``nice''  in the following sense. We say that an $\mathbb A$-graph  $\mathcal B$ is \emph{benign} if:
\begin{enumerate}
\item[(B0)] $\mathcal B$ represents $U$, that is,   $U(\mathcal B, u_0)=U$ for some vertex $u_0\in VB$ of type $v_0\in VA$. 

\item[(B1)] the underlying graph of $\mathcal B$ is a finite tree. 
 
\item[(B2)] the edge groups are either trivial or equal to $ \langle m_e\rangle$.    

\item[(B3)] for  each  vertex $u\in VB$ of type $v_0$  the  corresponding vertex group    $B_u$ is meridional (generated by conjugates in $A_{v_0}$ of $x_1$).

\item[(B4)] for each vertex $u\in VB$ of type $v_1$  the following hold:
\begin{enumerate}
\item  $B_u=\langle  g_1m_1g_1^{-1}, \ldots  , g_rm_1g_r^{-1}\rangle$  
such that $  r= w(B_u) < b(\mathfrak k_1)$. 

\item for each $i=1, \ldots , r$ there is an edge $f_i\in   Star_B(u)\cap E_{\mathcal B}$ such that  $g_im_1g_i^{-1}$ is (in $B_u$) conjugate to $ o_{f_i} m_1 o_{f_i}^{-1}$.
\end{enumerate} 
\end{enumerate}

The complexity of a benign $\mathbb A$-graph $\mathcal B$ is defined as  the triple
$$ c(\mathcal B):= (c_1(\mathcal B), |EB| , |EB|-|E_{\mathcal B}|)  $$
where 
$$c_1(\mathcal B)= \sum_{\substack{ u\in VB  \\  [u]=v_0}} rank(B_u).$$

The same construction as in the previous section shows that there is a benign  $\mathbb A$-graph such that $c_1(\mathcal B_0)=r=w(U)$. Now choose a benign $\mathbb A$-graph $\mathcal B$ such that $c(\mathcal B)$ is minimal. An argument  completely analogous  to the one given in the  previous section,  with the obvious adjustments and using Lemma~\ref{C1} instead of Proposition~\ref{prop:01} and the notion of good subgroups, shows that if   $\mathcal B$ is not folded then there is a benign $\mathbb A$-graph that represents $U$ and has smaller complexity.  Therefore $\mathcal B$  must be folded.

We now use the foldedness  of $\mathcal B$ to show that $U$ is tame.  According to Lemma~\ref{C1}, for each $u\in VB$ with $[u]=v_0$, the corresponding vertex group $B_u$ is freely generated by $$g_{u,1}x_1g_{u,1}^{-1}, \ldots , g_{u,r_u} x_1 g_{u,r_u}^{-1}$$ where $r_u:=rank(B_u)$ and $g_{u,i}\in A_{v_0}$.  For each  such $u$ let $\gamma_u$ be a  reduced $\mathbb B$-path  from $u_0$ to $u$.  Then 
$$S:=\{m_{u,i}:=[ \mu( \gamma_u \cdot g_{u,i} x_1 g_{u,i}^{-1}\cdot \gamma_u^{-1})]  \ | \ u\in VB \ \text{  s.t. }   \ [u]=v_0 \ \text{ and } \   1\le i\le r_u\}$$
is a meridional generating set of $U$ of minimal size.  First observe that, as $\mathcal B$ is folded, $m_{u,i}$ is conjugate to $m_{u',i'}$ iff $u=u'$ and $g_{u,i} x_1 g_{u, i}^{-1}$  is in $B_u$  cojugate to $g_{u, i'} x_1g_{u,i'}^{-1}$. Lemma~\ref{C1} therefore implies that $i=i'$ and so the meridians $m_{u,i}$ are pairwise non conjugate. Next we  show that  for any $g\in G(\mathfrak k)$  either $gP(\mathfrak k)g^{-1} \cap U=1$ or $gP(\mathfrak k)g^{-1} \cap U= g\langle m\rangle g^{-1}$ and $gmg^{-1}$ is in $U$ conjugate to  some of the meridians  in $S$.  Assume that   $gpg^{-1}\in U$ for some non-trivial   peripheral element $p\in P(\mathfrak k)$.  The foldedness of $\mathcal B$ implies that there is $u\in VB$ of type $v_0$ such that $B_u$ contains a conjugate (in $A_{v_0}$) of $p$. Lemma~\ref{C1} implies that:
\begin{enumerate}
\item $p\in \langle x_1\rangle$ and   $cx_1c^{-1}\in  B_u$  for some $c\in A_{v_0}$.

\item $cx_1c^{-1}$ is in $B_u$ conjugate to $g_{u,i} x_1 g_{u,i}^{-1}$ for some $1\le i\le r_u$.
\end{enumerate}
Therefore $gP(\mathfrak k)g^{-1} = g\langle m \rangle g^{-1}$ and $gmg^{-1}$ is in $U$ conjugate $m_{u,i}$ for some $u$ and some $i$ as before. This completes the proof  of the theorem.


\bibliographystyle{amsplain}

\end{document}